\documentclass[a4paper,10pt,reqno]{amsart}

\usepackage{pstricks}
\usepackage{graphicx}
\usepackage{psfrag}
\usepackage{fancybox}
\usepackage[centertags]{amsmath}
\usepackage{amsfonts,amssymb,amsthm}
\usepackage{float,longtable}
\usepackage[centertags]{amsmath}
\usepackage{amsfonts,amssymb,amsthm}
\usepackage{float,longtable}
\usepackage{graphicx}
\usepackage{a4wide}
\usepackage[latin1]{inputenc}
\usepackage{newlfont}

\newtheorem{Theorem}{Theorem} 
\newtheorem{Proposition}[Theorem]{Proposition}
\newtheorem{Remark}[Theorem]{Remark}
\newtheorem{Lemma}[Theorem]{Lemma}

\numberwithin{Theorem}{section}
\numberwithin{equation}{section}

\title[Adjoint methods for Obstacle problems and Weakly coupled systems of PDE]{Adjoint methods for Obstacle problems and Weakly coupled systems of PDE}
\author[F. Cagnetti, D. Gomes, H. V. Tran]{Filippo Cagnetti, Diogo Gomes, Hung Vinh Tran}
\address[F. Cagnetti]{Departamento de Matem\'atica
Instituto Superior T\'ecnico,
Av. Rovisco Pais, 1049-001 Lisboa, Portugal}
\email{cagnetti@math.ist.utl.pt}
\address[D. Gomes]{Departamento de Matem\'atica
Instituto Superior T\'ecnico,
Av. Rovisco Pais, 1049-001 Lisboa, Portugal}
\email{dgomes@math.ist.utl.pt}
\address[H. V. Tran]{Department of Mathematics,
University of California Berkeley, CA, 94720-3840}
\email{tvhung@math.berkeley.edu}

\begin{document}

\date{} % delete this line to display the current date

\begin{abstract}
The adjoint method, recently introduced by Evans, 
is used to study obstacle problems, weakly coupled systems, 
cell problems for  weakly coupled systems of Hamilton--Jacobi equations, 
and weakly coupled systems of obstacle type.
In particular, new results about the speed of convergence of some approximation procedures are derived. 
\end{abstract}

\maketitle

%%%%%
{\small

\bigskip

}

%%%%%%

\begin{section}{Introduction}

{We study} the speed of convergence of certain approximations
for obstacle problems and weakly coupled systems of Hamilton--Jacobi equations,
using the Adjoint Method.
This technique, recently introduced by Evans 
(see \cite{E2}, and also \cite{T1} and \cite{CGT1}),
is a very successful tool to understand several types of degenerate PDEs.
It can be applied, for instance, to Hamilton--Jacobi equations with non convex Hamiltonians, 
e.g. time dependent (see \cite{E2}) and time independent (see \cite{T1}),
to weak KAM theory (see \cite{CGT1}), and to the infinity Laplacian equation (see \cite{E3}). 
We address here several new applications, and propose some new open questions. 
Further results, which will not be discussed here, can be found in \cite{E2} and \cite{CGT1}.

\begin{subsection}{Outline of the paper} 
{The} paper contains four further sections concerning 
obstacle problems, weakly coupled systems, effective Hamiltonian 
for weakly coupled systems of Hamilton--Jacobi equations, 
and weakly coupled systems of obstacle type, respectively.
We use a common strategy to study all these problems.
Note, however, that each of them presents different challenges, which
are described in the corresponding sections.
Also, we believe that the applications we present here illustrate 
how to face the difficulties that can be encountered in the study 
of other systems of PDEs and related models.
In particular, we show how to control singular terms arising from the switching
to an obstacle (Lemma \ref{obs_lem2}), random switching (Lemma \ref{kiz}), 
or optimal switching (Lemma \ref{gamprimeb}).

\medskip

In order to clarify our approach, let us give the details
of its application to the obstacle problem (see Section \ref{sectobs}):
\begin{equation} \label{qw}
\left\{ \begin{aligned}
\max \{ u- \psi, u+H(x,Du)\}&=0 \quad \mbox{in}~ U,\vspace{.05in}\\
u &= 0 \quad \mbox{on}~ \partial U,\\
\end{aligned} \right. 
\end{equation}
where $\psi: \overline{U} \to \mathbb{R}$ and $H: \mathbb R^n \times \overline{U} \to \mathbb R$ are smooth,
with $\psi \geq 0$ on $\partial U$. 
Here and in all the paper, $U$ is an open bounded domain in $\mathbb R^n$ with smooth boundary, 
and $n \geq 2$.
Moreover, we will denote with $\nu$ the \textit{outer} unit normal to $\partial U$.
This equation arises naturally in optimal control theory, in the study of optimal stopping (see \cite{L}). 
See also \cite{BP,IY}.
%in the special case in which $H$ is uniformly convex in the second variable.

Classically, in order to study \eqref{qw} one first modifies the equation, 
by adding a perturbation term that penalizes the region where $u > \psi$.
Then, a solution is obtained as a limit of the solutions of the penalized problems.
More precisely, let $\gamma: \mathbb R \to \mathbb [0, +\infty)$ be smooth
such that
\begin{equation}
\left\{ \begin{aligned}
\gamma(s)=0~\mbox{for}~s\le 0, \quad \gamma(s)>0~\mbox{for}~s>0,\vspace{.05in}\\
0<\gamma'(s) \le 1~\mbox{for}~s>0,~\mbox{and}~ \lim_{s \to + \infty} \gamma(s)= + \infty,\\
\end{aligned} \right. 
\notag
\end{equation}
and define $\gamma^\varepsilon: \mathbb R \to \mathbb [0, +\infty)$ as
\begin{equation} \label{gammaep}
\gamma^\varepsilon(s) := \gamma \left( \dfrac{s}{\varepsilon} \right),~\mbox{for all}~s \in \mathbb R, 
\quad \mbox{for all}~\varepsilon > 0.
\end{equation}
In some of the problems we discuss we also require $\gamma$ to
be convex in order to obtain improved results, but that will be pointed out where necessary.
For every $\varepsilon >0$, one can introduce the penalized PDE
\begin{equation}
\left\{ \begin{aligned}
u^\varepsilon+H(x,Du^\varepsilon)+\gamma^\varepsilon(u^\varepsilon-\psi)
&= \varepsilon \Delta u^\varepsilon \quad \mbox{in}~ U, \vspace{.05in}\\
u^\varepsilon &= 0 \qquad \mbox{on}~ \partial U.\\
\end{aligned} \right. 
\label{obs_regintr}
\end{equation}
To avoid confusion, we stress the fact that here $\gamma^\varepsilon(u^\varepsilon-\psi)$
stands for the composition of the function $\gamma^{\varepsilon}$ with $u^\varepsilon-\psi$.
Unless otherwise stated, we will often simply write 
$\gamma^{\varepsilon}$ and $\left( \gamma^{\varepsilon} \right)'$
to denote $\gamma^\varepsilon(u^\varepsilon-\psi)$ 
and $\left( \gamma^\varepsilon \right)' (u^\varepsilon-\psi)$, respectively.

Thanks to \cite{L}, for every $\varepsilon > 0$ 
there exists a smooth solution $u^\varepsilon$ to \eqref{obs_regintr}.
It is also well known that, up to subsequences, $u^{\varepsilon}$
converges uniformly to a viscosity solution $u$ of \eqref{qw} 
(see also Section \ref{sectobs} for further details).

%In \textcolor{blue}{\cite{BP} Barles and Perthame considered related problems.}
%when $H ( x , \cdot)$ is uniformly convex, and studied
%the speed of convergence of the functions $u^{\varepsilon}$ to $u$.
%However, both the original problem,
%the regularized PDE, and their methods are different from ours.
%To the best of our knowledge, no results are available in literature 
%concerning non convex Hamiltonians.

We face here the problem requiring a coercivity assumption 
on $H$ and a compatibility condition for equation \eqref{qw} (see hypotheses (H\ref{sectobs}.1) and (H\ref{sectobs}.2), respectively), 
and show that the speed of convergence in the general case is $O(\varepsilon^{1/2})$.
Notice that we do not require the Hamiltonian $H$ to be convex in $p$.
\begin{Theorem}
\label{obs_speed}
Suppose conditions (H\ref{sectobs}.1) and (H\ref{sectobs}.2) in Section~\ref{sectobs} hold.
Then, there exists a positive constant $C$, independent of
$\varepsilon$, such that
\begin{equation}
\|u^\varepsilon - u\|_{L^\infty} \le C \varepsilon^{1/2}.
\notag
\end{equation}
\end{Theorem}
%At the end of the section, we give a dynamic and a stochastic interpretation
%of the problem (see Subsection \ref{dyninter} and Subsection \ref{stocinter}, respectively).
{The proof of Theorem \ref{obs_speed} consists of three steps.}
\smallskip

\textbf{Step I: Preliminary Estimates.} We first show that
\begin{equation} \label{hfj}
\max_{x \in \overline{U}} \frac{u^{\varepsilon} (x) - \psi (x)}{\varepsilon} \leq C, 
\end{equation}
for some constant $C > 0$ independent of $\varepsilon$ (see Lemma \ref{obs_lem4}).
This allows us to prove that 
\begin{equation*}
 \| u^{\varepsilon} \|_{L^{\infty}}, \| D u^{\varepsilon} \|_{L^{\infty}} \leq C,
\end{equation*}
see Proposition \ref{obs1}.

\smallskip

\textbf{Step II: Adjoint Method.}
We consider the formal linearization of \eqref{obs_regintr}, 
and then introduce the correspondent adjoint equation (see equation~\eqref{obs_adj}).
The study of this last equation for different values of the right-hand side
allows us to obtain several useful estimates (see Lemma~\ref{obs_lem2} and Lemma~\ref{obs_lem3}).

\smallskip

\textbf{Step III: Conclusion.}
We conclude the proof of Theorem \ref{obs_speed} by showing that
\begin{equation} \label{hf}
\max_{x \in \overline{U}} \left|  u^\varepsilon_{\varepsilon} (x)  \right| \le
\dfrac{C}{\varepsilon^{1/2}}, \qquad \qquad u^\varepsilon_{\varepsilon} (x) := \frac{\partial u^\varepsilon}{\partial \varepsilon} (x),
\end{equation}
for some constant $C > 0$ independent of $\varepsilon$ (see Lemma \ref{obs_thm1}).
The most delicate part of the proof of \eqref{hf} consists in controlling the term 
(see relation \eqref{obsk}) 
\begin{equation} \label{dvar}
\gamma^{\varepsilon}_{\varepsilon} (s)
:= \dfrac{\partial \gamma^\varepsilon }{\partial \varepsilon} (s)
 =-\dfrac{s}{\varepsilon^2} 
\gamma ' \left( \dfrac{s}{\varepsilon} \right),
\qquad \text{for } s \in \mathbb{R} .
\end{equation}

\medskip

We underline that getting a bound for \eqref{dvar} 
can be extremely hard in general.
In this context, this is achieved by differentiating equation \eqref{obs_regintr} w.r.t. $\varepsilon$
(see equation \eqref{obs5}), and then by using {inequality} \eqref{hfj}, 
Lemma \ref{obs_lem2} and Lemma \ref{obs_lem3}.
This means that we overcome the problem by essentially using 
the Maximum Principle and the monotonicity of $\gamma^\varepsilon$
(see estimates \eqref{obs6} and \eqref{uj}).
We were not able to obtain such a bound
when dealing with homogenization or singular perturbation, where also similar terms appear. 
We believe it would be very interesting to find the correct way 
to apply the Adjoint Method in these situations.

\medskip

In Section \ref{weakcou} we study
monotone weakly coupled systems of Hamilton--Jacobi equations
\begin{equation}
\left\{ \begin{aligned}
c_{11} u_1 +c_{12} u_2 + H_1(x,Du_1) &=0, \vspace{.05in}\\
c_{21} u_1 +c_{22} u_2 + H_2(x,D u_2) &=0, \\
\end{aligned} \right. 
\quad \mbox{in}~ U,
\label{syst}
\end{equation}
with boundary conditions $u_1 = u_2=0$ on $\partial U$,
by considering the following approximation:
\begin{equation*}
\left\{ \begin{aligned}
c_{11} u^\varepsilon_1 +c_{12} u^\varepsilon_2 + H_1(x,D u^\varepsilon_1) &=\varepsilon \Delta u^\varepsilon_1
\vspace{.05in}\\
c_{21}u^\varepsilon_1+c_{22}u^\varepsilon_2 + H_2(x,Du^\varepsilon_2) &=\varepsilon \Delta u^\varepsilon_2 \\
\end{aligned} \right. 
\quad \mbox{in}~ U,
%\label{wc_reg}
\end{equation*}
with $u^\varepsilon_1 = u^\varepsilon_2=0$ on $\partial U$.
\noindent
Under some coupling assumptions on the coefficients
(see conditions (H\ref{weakcou}.2) and (H\ref{weakcou}.3)),
Engler and Lenhart \cite{EL}, Ishii and Koike \cite{IK1} 
prove existence, uniqueness and stability 
for the viscosity solutions  $(u_1,u_2)$ of \eqref{syst}, but they do not consider any approximation of the system.
We observe that these coupling assumptions
are similar to monotone conditions of single equations,
and play a crucial role in the establishment of  
the comparison principle and uniqueness result, and thus cannot be removed.

We show that, 
under the same assumptions of \cite{EL}, the speed of convergence 
of $(u^{\varepsilon}_1,u^{\varepsilon}_2)$
to  $(u_1,u_2)$ is $O(\varepsilon^{1/2})$ (see Theorem \ref{wc_thm1}).
For the sake of simplicity, we just focus on a system of two equations, 
but the general case can be treated in a similar way.

\medskip

Section \ref{secweak} is devoted to an analog of the {\it cell problem}
introduced by Lions, Papanicolaou and Varadhan \cite{LPV}.
More precisely, we consider the following quasi-monotone
weakly coupled system of Hamilton--Jacobi equations:
\begin{equation}
\left\{ \begin{aligned}
c_1 u_1- c_1 u_2 + H_1(x,Du_1) &=\overline{H}_1 \vspace{.05in}\\
-c_2 u_1 + c_2 u_2 + H_2(x,Du_1) &=\overline{H}_2 \\
\end{aligned} \right. 
\quad \mbox{in}~ {\mathbb T^n},
\label{tripz}
\end{equation}
also called the \textit{cell problem}.
Here $c_{1}$ and $c_{2}$ are positive constants
and $H_1, H_2 : \mathbb T^n \times \mathbb{R}^n \to \mathbb{R}$ are smooth,
while $u_1, u_2 : \mathbb T^n \to \mathbb R$ and $\overline{H}_1,\overline{H}_2 \in \mathbb{R}$ are unknowns.
Systems of this type have been studied by Camilli, Loreti and Yamada in \cite{CL} and \cite{CLY}, 
for uniformly convex Hamiltonians in a bounded domain. 
They arise naturally in optimal control and in large deviation theory for random
evolution processes.
Under a coercivity-like assumption on $H_1, H_2$ (see condition (H\ref{secweak}.1)),
we obtain the following new result.
\begin{Theorem} \label{cell_const}
Assume that (H\ref{secweak}.1) holds.
Then, there exists a pair of constants $(\overline{H}_1, \overline{H}_2)$ 
such that \eqref{tripz} admits a viscosity solution $(u_1,u_2) \in C(\mathbb T^n)^2$.
\end{Theorem}
One can easily see that the pair $(\overline{H}_1, \overline{H}_2)$ is not unique (see Remark \ref{nouniq}).
Nevertheless, we have the following. 
\begin{Theorem} \label{cell_const2}
There exists a unique $\mu \in \mathbb R$ such that
\begin{equation*}
c_2 \overline{H}_1 + c_1\overline{H}_2 = \mu,
\end{equation*}
for every pair $(\overline{H}_1, \overline{H}_2) \in \mathbb{R}^2$ such that 
\eqref{tripz} admits a viscosity solution $(u_1,u_2) \in C(\mathbb T^n)^2$.
\end{Theorem}
Theorem \ref{cell_const2} can be rephrased  
by saying that there exists a unique $\overline{H} \in \mathbb{R}$ such that the system 
\begin{equation} \label{falsj}
\left\{ \begin{aligned}
c_1 u_1- c_1 u_2 + H_1(x,Du_1) &=\overline{H} \vspace{.05in}\\
-c_2 u_1 + c_2 u_2 + H_2(x,Du_1) &=\overline{H} \\
\end{aligned} \right. 
\quad \mbox{in}~ {\mathbb T^n},
\end{equation}
admits viscosity solutions $u_1,u_2 \in C(\mathbb T^n)$, with (see Remark \ref{Hbar})
$$
\overline{H}= \frac{\mu}{c_1+c_2} .
$$
Thus, this is the analogous to the uniqueness result of the effective Hamiltonian for the single equation case in \cite{LPV}.
Notice that this {\it cell problem} is the important basis for the study 
of homogenization and
large time behavior of weakly coupled systems
of Hamilton--Jacobi equations.

Besides, we also consider the regularized system
\begin{equation} 
\left\{ \begin{aligned}
(c_1+\varepsilon)u^\varepsilon_1 - c_1 u^\varepsilon_2 + H_1(x,Du^\varepsilon_1)
& = \varepsilon^2 \Delta u^\varepsilon_1
\vspace{.05in}\\
(c_2+\varepsilon) u^\varepsilon_2 - c_2 u^\varepsilon_1 + H_2(x,D u^\varepsilon_2 )
&=\varepsilon^2 \Delta u^\varepsilon_2 \\
\end{aligned} \right. 
\quad \mbox{in}~ \mathbb T^n,
%\label{cell_reg}
\label{trip}
\end{equation}
and prove that both $\varepsilon u_1^{\varepsilon}$ and $\varepsilon u_2^{\varepsilon}$ 
converge uniformly to $-\overline{H}$ with speed of convergence $O (\varepsilon)$ (see Theorem~\ref{cell_thm2}).
We call $\overline{H}$ the effective Hamiltonian of the {\it cell problem}
for {the} weakly coupled system of Hamilton--Jacobi equations.

\medskip
In Section \ref{kilo}, we conclude the paper with the study of weakly coupled systems of obstacle type, namely
\begin{equation} \label{jk}
\left\{ \begin{aligned}
\max\{ u_1 - u_2 -\psi_1 , u_1 +H_1(x,D u_1) \}&=0 \quad\mbox{in}~ U, \vspace{.05in}\\
\max\{ u_2-u_1-\psi_2 , u_2+H_2(x,Du_2) \}&=0 \quad\mbox{in}~ U, \\
\end{aligned} \right.
\end{equation}
{with boundary conditions $u_1 = u_2=0$ on $\partial U$.}
Problems of this type appeared in \cite{DE} and \cite{CLY}.
Here $H_1, H_2: \overline{U} \times \mathbb{R}^n \to \mathbb{R}$ 
and $\psi_1, \psi_2 : \overline{U} \to \mathbb{R}$ are smooth, with $\psi_1, \psi_2 \geq \alpha > 0$.

In this case, although the two equations in \eqref{jk} are coupled just through the difference $u_1~-~u_2$
(\textit{weakly} coupled system), the problem turns out to be considerably more difficult 
than the corresponding scalar equation \eqref{qw}.
Indeed, we cannot show now the analogous of estimate \eqref{hfj}
as in Section~\ref{sectobs}.
For this reason, the hypotheses we require are stronger than in the scalar case.
Together with the usual hypotheses of coercivity and compatibility
(see conditions (H\ref{kilo}.2) and (H\ref{kilo}.4)), 
we have to assume that $H_1 (x,\cdot)$ and $H_2 (x,\cdot)$ are convex (see (H\ref{kilo}.1)),
and we also ask that $D_x H_1$ and $D_x H_2$ are bounded (see (H\ref{kilo}.3)).
Under these hypotheses, that are natural in optimal switching problems, 
we are able to establish several delicate estimates 
by crucially employing the adjoint method (Lemmas from \ref{gamprimeb} to \ref{dggt}), 
which then yield a rate of convergence (Theorem \ref{obs_speedfin}).
\end{subsection}

\begin{subsection}{Acknowledgments} 

The authors are grateful to Craig Evans and 
Fraydoun Rezakhanlou for very useful discussions on the subject of the paper.  
We thank the anonymous referee for many valuable comments and suggestions.
Diogo Gomes was partially supported by CAMGSD/IST through FCT Program POCTI - FEDER
and by grants PTDC/EEA-ACR/67020/2006, PTDC/MAT/69635/2006, and UTAustin/MAT/0057/2008. 
Filippo Cagnetti was partially supported by FCT through the CMU$|$Portugal program.
Hung Tran was partially supported by VEF fellowship.

\end{subsection}

\end{section}

\begin{section}{Obstacle problem} \label{sectobs}

In this section, we study the following obstacle problem
\begin{equation}
\left\{ \begin{aligned}
\max \{ u- \psi, u+H(x,Du)\}&=0 \quad \mbox{in}~ U\vspace{.05in}\\
u &= 0 \quad \mbox{on}~ \partial U,\\
\end{aligned} \right. 
\label{obs_eqn}
\end{equation}
where $\psi: \overline{U} \to \mathbb{R}$ and $H: {\overline{U}  \times \mathbb R^n} \to \mathbb R$
are smooth, with $\psi \ge 0$ on $\partial U$.
We also assume that
\begin{itemize}
\vspace{.1cm}
 \item[(H\ref{sectobs}.1)] {there exists $\beta >0$ such that}
$$
\lim_{|p| \to + \infty}  
\left( {\beta}|H(x,p)|^2 + D_x H(x,p) \cdot p \right)= 
\lim_{|p| \to + \infty} \frac{H(x,p)}{|p|}=+\infty 
\text{ uniformly in } x \in \overline U;
$$
%\vspace{.2cm}
 \item[(H\ref{sectobs}.2)] there exists a function $\Phi \in C^2 (U) \cap C^1 (\overline U)$ 
 such that $\Phi \leq \psi$ on $\overline U$, $\Phi=0$ on $\partial U$ and
 $$
 \Phi+H(x,D\Phi) <0\quad \mbox{in}~ \overline{U}.
 $$
\vspace{.1cm}
\end{itemize}
We observe that in the classical case $H(x,p) = \mathcal{H} (p)+V(x)$ with 
$$
\lim_{|p| \to + \infty} \frac{\mathcal{H} (p)}{|p|} = + \infty,
$$  
or when $H$ is superlinear in $p$ and $|D_xH(x,p)| \le C(1+|p|)$, then we immediately have (H\ref{sectobs}.1).
Assumption (H\ref{sectobs}.2) (stating, in particular, that $\Phi$ is a sub-solution of \eqref{obs_eqn}), 
will be used to derive the existence of solutions of \eqref{obs_eqn}, 
and to give a uniform bound for the gradient of solutions of the penalized equation below.
\begin{subsection}{The classical approach}
 
% Classically, in order to study \eqref{obs_eqn} one first
% modifies the equation by adding a perturbation term 
% that penalizes the regions where $u > \psi$,
% and then passes to the limit as the perturbation goes to zero.
% More precisely, let $\gamma: \mathbb R \to \mathbb [0, +\infty)$ be smooth and convex, 
% such that
% \begin{equation}
% \left\{ \begin{aligned}
% \gamma(s)=0~\mbox{for}~s\le 0, \quad \gamma(s)>0~\mbox{for}~s>0,\vspace{.05in}\\
% 0<\gamma'(s) \le 1~\mbox{for}~s>0,~\mbox{and}~ \lim_{s \to + \infty} \gamma(s)= + \infty,\\
% \end{aligned} \right. 
% \notag
% \end{equation}
% and define $\gamma^\varepsilon: \mathbb R \to \mathbb [0, +\infty)$ as
% $$
% \gamma^\varepsilon(s) := \gamma \left( \dfrac{s}{\varepsilon} \right),~\mbox{for all}~s \in \mathbb R.
% $$
For every $\varepsilon >0$, the \textit{penalized} equation
of \eqref{obs_eqn} is given by
\begin{equation}
\left\{ \begin{aligned}
u^\varepsilon+H(x,Du^\varepsilon)+\gamma^\varepsilon(u^\varepsilon-\psi)
&= \varepsilon \Delta u^\varepsilon \quad \mbox{in}~ U, \vspace{.05in}\\
u^\varepsilon &= 0 \qquad \mbox{on}~ \partial U,\\
\end{aligned} \right. 
\label{obs_reg}
\end{equation}
where $\gamma^{\varepsilon}$ is defined by \eqref{gammaep}.
From \cite{L} it follows that under conditions (H\ref{sectobs}.1) and (H\ref{sectobs}.2), 
for every $\varepsilon > 0$ there exists
a smooth solution $u^\varepsilon$ to \eqref{obs_reg}.
The first result we establish is a uniform bound 
for the $C^1$-norm of the sequence $\{u^{\varepsilon}\}$.
\begin{Proposition} \label{obs1}
There exists a positive constant $C$, independent of $\varepsilon$, such that
\begin{equation*}
\|u^\varepsilon\|_{L^\infty}, \|Du^\varepsilon\|_{L^\infty} \le C.
\end{equation*}
\end{Proposition}
In order to prove Proposition \ref{obs1}, we need the following fundamental lemma:
\begin{Lemma} \label{obs_lem4}
There exists a constant $C>0$, independent of $\varepsilon$, such that
\begin{equation}
\max_{x \in \overline{U}} \gamma^\varepsilon (u^\varepsilon -\psi) \le C, \quad \quad \quad \max_{x \in \overline{U}}
\dfrac{u^\varepsilon -\psi}{\varepsilon} \le C.
\notag
\end{equation}
\end{Lemma}
\begin{proof}
We only need to show that $\max_{x \in \overline{U}} \gamma^\varepsilon (u^\varepsilon -\psi) \le C$,
since then the second estimate follows directly by the definition of $\gamma^\varepsilon$.
Since $u^\varepsilon - \psi \le 0$ on $\partial U$, we have $\max_{x \in \partial U} \gamma^\varepsilon
(u^\varepsilon -\psi) = 0$.

Now, if $\max_{x \in \overline{U}} \gamma^\varepsilon (u^\varepsilon -\psi)=0$, then we are done.
Thus, let us assume that there exists $x_1 \in U$ such that 
$\max_{x \in \overline{U}} \gamma^\varepsilon (u^\varepsilon -\psi)
=\gamma^\varepsilon (u^\varepsilon -\psi)(x_1) >0$.
Since $\gamma^\varepsilon$ is increasing, we also have 
$\max_{x \in U} (u^\varepsilon -\psi) = u^\varepsilon (x_1)- \psi (x_1)$.
Thus, using \eqref{obs_reg}, by the Maximum principle
\begin{align*} 
&(u^\varepsilon(x_1) - \psi(x_1)) + \gamma^\varepsilon (u^\varepsilon(x_1) - \psi(x_1)) 
= \varepsilon \Delta u^{\varepsilon} (x_1) - H(x_1, D u^{\varepsilon} (x_1)) - \psi(x_1) \\
& \le \varepsilon \Delta \psi(x_1)  - H(x_1, D \psi(x_1)) - \psi(x_1).
\end{align*}
Since $u^\varepsilon(x_1) - \psi(x_1)>0$,
\begin{equation}
\gamma^\varepsilon (u^\varepsilon(x_1) - \psi(x_1)) \le \max_{x \in \overline{U}} (|\Delta \psi| +  |H(x,D \psi)|+
|\psi(x)|) \le C,
\notag
\end{equation}
for any $\varepsilon<1$, and this concludes the proof.
\end{proof}

\begin{proof}[Proof of Proposition \ref{obs1}]
Suppose there exists $x_0 \in U$ such that $u^{\varepsilon} (x_0) = \max_{x \in \overline{U}} u^{\varepsilon} (x)$.
Then, since $\Delta u^\varepsilon (x_0)\leq~0$ and using the fact that $\gamma^{\varepsilon} \geq 0$
\begin{equation*} \begin{split}
u^\varepsilon (x_0) &= \varepsilon \Delta u^\varepsilon (x_0)
- H(x_0, 0 ) - \gamma^\varepsilon(u^\varepsilon (x_0) -\psi (x_0)) \\
&\leq - H(x_0, 0 ) \leq \max_{x \in \overline{U}} \left( - H(x, 0 ) \right) \leq C.
\end{split} \end{equation*}
Let now $x_1 \in U$ be such that $u^{\varepsilon} (x_1) = \min_{x \in \overline{U}} u^{\varepsilon} (x_1)$.
Then, using Lemma~\ref{obs_lem4},
\begin{equation*} \begin{split}
u^\varepsilon (x_1) &= \varepsilon \Delta u^\varepsilon (x_1)
- H(x_1, 0 ) - \gamma^\varepsilon(u^\varepsilon (x_1) -\psi (x_1)) \\
&\geq - H(x_1, 0 ) - \gamma^\varepsilon(u^\varepsilon (x_1) -\psi (x_1)) \\
&\geq \min_{x \in \overline{U}} 
\left( - H(x, 0 ) - \gamma^\varepsilon(u^\varepsilon (x) -\psi (x)) \right) \geq -C.
\end{split} \end{equation*}
This shows that $\| u^{\varepsilon} \|_{L^{\infty}}$ is bounded.

To prove that $\| Du^{\varepsilon} \|_{L^{\infty}}$ is bounded independently of $\varepsilon$, 
we first need to prove that $\| Du^{\varepsilon} \|_{L^{\infty}(\partial U)}$ is bounded
by constructing appropriate barriers.

Let $\Phi$ be as in (H2.2). For $\varepsilon$ small enough, we have that
$$
\Phi+H(x,D\Phi)+\gamma^\varepsilon(\Phi-\psi) <\varepsilon \Delta \Phi,
$$
and $\Phi=0$ on $\partial U$. 
Therefore, $\Phi$ is a sub-solution of \eqref{obs_reg}. By the comparison principle, $u^\varepsilon \ge \Phi$ in~$U$.

Let now $d(x)=\text{dist}(x,\partial U)$. 
It is well-known that for some $\delta>0$ $d \in C^2(U_\delta)$ and $|Dd|=1$ in $U_\delta$, 
where $U_\delta :=\{x\in U:~d(x) <\delta\}$.
For $\mu>0$ large enough, the uniform bound on $\|u^\varepsilon\|_{L^\infty}$ yields
$v:=\mu  d \ge u^\varepsilon$ on $\partial U_\delta $. Assumption (H\ref{sectobs}.1) then implies
$$
v+H(x,Dv)+\gamma^\varepsilon(v-\psi) - \varepsilon \Delta v \ge H(x,\mu Dd) -C \mu \ge 0,
$$ 
for  $\mu$ is sufficiently large.
So the comparison principle gives us that $\Phi \le u^\varepsilon \le v$ in $ U_\delta$. 
Thus, since $\nu$ is the \textit{outer} unit normal to $\partial U$, and $\Phi= u^\varepsilon= v = 0$ on $\partial U$, we have
$$
\dfrac{\partial v}{\partial \nu}(x) \le \dfrac{\partial u^\varepsilon}{\partial \nu}(x) \le \dfrac{\partial \Phi}{\partial \nu}(x),\quad \text{for } x\in \partial U.
$$
Hence, we obtain $\| Du^{\varepsilon} \|_{L^{\infty}(\partial U)} \leq C$.

Next, let us set $w^\varepsilon = \dfrac{|Du^{\varepsilon}|^2}{2}$. 
By a direct computation one can see that
\begin{equation} \label{obs2}
2(1+(\gamma^\varepsilon) ' )w^\varepsilon + D_p H \cdot D w^\varepsilon 
+ D_xH\cdot Du^\varepsilon-(\gamma^\varepsilon)' Du^\varepsilon\cdot D\psi = \varepsilon \Delta w^\varepsilon -
\varepsilon |D^2 u^\varepsilon|^2.
\end{equation}
If $\| Du^{\varepsilon} \|_{L^{\infty}} \leq \max (\| D \psi \|_{L^{\infty}},\| Du^{\varepsilon} \|_{L^{\infty}(\partial U)})$ then we are done.\\
Otherwise, $ \max (\| D \psi \|_{L^{\infty}},\| Du^{\varepsilon} \|_{L^{\infty}(\partial U)}) < \| D u^{\varepsilon} \|_{L^{\infty}}$.
We can choose $x_2 \in U$ such that $w^{\varepsilon} (x_2) = \max_{x \in \overline{U}} w^{\varepsilon} (x)$.
Then, using \eqref{obs2}
\begin{equation} \label{1k} \begin{split}
\varepsilon |D^2 u^\varepsilon|^2  (x_2)
&= \varepsilon \Delta w^\varepsilon (x_2)
- 2 w^\varepsilon (x_2) - D_xH (x_2 , Du^\varepsilon (x_2)) \cdot Du^\varepsilon (x_2) \\
&+ (\gamma^\varepsilon)' 
\left( Du^\varepsilon (x_2) \cdot D\psi (x_2) - |D u^{\varepsilon}|^2 (x_2) \right) \\
&\leq - D_xH (x_2 , Du^\varepsilon (x_2)) \cdot Du^\varepsilon (x_2).
\end{split} \end{equation}
Moreover, for $\varepsilon$ sufficiently small we have
$$
 \varepsilon |D^2 u^\varepsilon|^2  (x_2)
\geq 2 \beta \varepsilon^2 |\Delta u^{\varepsilon}  (x_2)|^2
= 2 \beta \left[ u^\varepsilon  (x_2) + H(x_2 ,Du^\varepsilon  (x_2))
+\gamma^\varepsilon(u^\varepsilon  (x_2) - \psi  (x_2)) \right]^2.
$$
{By Young's inequality,
\begin{equation} \label{2k} \begin{split}
\varepsilon |D^2 u^\varepsilon|^2  (x_2)
&\geq \beta |  H(x_2 ,Du^\varepsilon  (x_2)) |^2 
-  2 \beta  \left[ u^\varepsilon  (x_2) +  \gamma^\varepsilon(u^\varepsilon  (x_2) - \psi  (x_2)) \right]^2 \\
&\geq \beta |  H(x_2 ,Du^\varepsilon  (x_2)) |^2 - C_{\beta}, 
\end{split} \end{equation}
for some positive constant $C_{\beta}$, 
where we used Lemma \ref{obs_lem4} for the last inequality.
Collecting \eqref{1k} and \eqref{2k}
\begin{equation*}
\beta |  H(x_2 ,Du^\varepsilon  (x_2)) |^2 
+ D_xH (x_2 , Du^\varepsilon (x_2)) \cdot Du^\varepsilon (x_2) \leq C_{\beta}.
\end{equation*}
Recalling hypothesis (H\ref{sectobs}.1),} we must have
\begin{equation*}
\| D u^{\varepsilon} \|_{L^{\infty}} = | Du^\varepsilon  (x_2) | \leq C.
\end{equation*}

\end{proof}
Thanks to Proposition \ref{obs1} one can show that, up to subsequences, 
$u^{\varepsilon}$ converges uniformly 
to a viscosity solution $u$ of the obstacle problem (\ref{obs_eqn}).

\end{subsection}

\begin{subsection}{Proof of Theorem \ref{obs_speed}}

We now study the speed of convergence.

\noindent
To prove our theorem we need several steps.

{\bf Adjoint method:} 
The formal linearized operator $L^\varepsilon: C^2 (U) \to C (U)$
corresponding to \eqref{obs_reg} is given by
$$
L^\varepsilon z :=(1+(\gamma^\varepsilon)') z  + D_pH\cdot Dz -\varepsilon \Delta z.
$$
We will now introduce the adjoint PDE corresponding to $L^\varepsilon$.
Let $x_0 \in U$ be fixed.
We denote by $\sigma^\varepsilon$ the solution of:
\begin{equation}
\left\{ \begin{aligned}
(1+(\gamma^\varepsilon)')\sigma^\varepsilon-\mbox{div}(D_pH\sigma^\varepsilon)
&=\varepsilon \Delta \sigma^\varepsilon+ \delta_{x_0},\ \qquad \qquad&\mbox{in}~ U ,\vspace{.05in}\\
\sigma^\varepsilon &= 0, \qquad \qquad\qquad &\mbox{on}~ \partial U, \\
\end{aligned} \right. 
\label{obs_adj}
\end{equation}
where $\delta_{x_0}$ stands for the Dirac measure concentrated in $x_0$.
In order to show existence and uniqueness of $\sigma^\varepsilon$, 
we have to pass to a further adjoint equation.
Let $f \in C (U)$ be fixed. 
Then, we denote by $v$ the solution to
 \begin{equation}
\left\{ \begin{aligned}
(1+(\gamma^\varepsilon)')v+D_pH\cdot Dv&=\varepsilon \Delta v+f, \qquad \qquad&\mbox{in}~ U ,\vspace{.05in}\\
v &= 0, \qquad \qquad\qquad &\mbox{on}~ \partial U. \\
\end{aligned} \right. 
\label{obs3}
\end{equation}
When $f \equiv 0$, by using the Maximum Principle one can show that $v \equiv 0$
is the unique solution to \eqref{obs3}. 
Thus, by the Fredholm Alternative we infer that \eqref{obs_adj}
admits a unique solution $\sigma^\varepsilon$.
Moreover, one can also prove that $\sigma^\varepsilon \in C^\infty(U\setminus \{x_0\})$.
Some additional properties of $\sigma^\varepsilon$ are given by the following lemma.
\begin{Lemma}[Properties of $\sigma^\varepsilon$] \label{obs_lem2}
Let $\nu$ denote the outer unit normal to $\partial U$. Then, 
\begin{itemize}
\item[(i)] $\sigma^\varepsilon \ge 0$ on $\overline{U}$.  
In particular, $\dfrac{\partial \sigma^\varepsilon}{\partial \nu} \le 0$ on $\partial U$.
\item[(ii)] The following equality holds: 
\begin{equation}
\int_U (1+(\gamma^\varepsilon)') \, \sigma^\varepsilon \,dx =1+\varepsilon \int_{\partial U} \dfrac{\partial
\sigma^\varepsilon}{\partial \nu}\, dS.
\notag
\end{equation}
In particular,
$$
\varepsilon \int_{\partial U} \left| \dfrac{\partial \sigma^\varepsilon}{\partial \nu} \right| \,dS \le 1.
$$
\end{itemize}
\end{Lemma}
\begin{proof}
First of all, consider equation \eqref{obs3} and observe that 
\begin{equation} \label{vpos}
f \geq 0 \Longrightarrow v \geq 0.
\end{equation}
Indeed, assume $f \geq 0$ and let $\overline{x} \in \overline{U}$ be such that
\begin{equation*}
v (\overline{x}) = \min_{x \in \overline{U}} v (x).
\end{equation*}
We can assume that $\overline{x} \in U$, since otherwise clearly $v \geq 0$. 
Then, for every $x \in U$
\begin{equation*}
((1+(\gamma^\varepsilon)') v (\overline{x}) = \varepsilon \Delta v (\overline{x}) + f (\overline{x}) \geq 0,
\end{equation*}
and \eqref{vpos} follows, since $1+(\gamma^\varepsilon)' > 0$ .

Now, multiply equation \eqref{obs_adj} by $v$ and integrate by parts, obtaining 
$$
\int_U f \sigma^\varepsilon\, dx =  v(x_0).
$$
Taking into account \eqref{vpos}, from last relation we infer that
\begin{equation*}
\int_U f \sigma^\varepsilon\, dx \geq 0 \qquad \text{ for every }f \geq 0, 
\end{equation*}
and this implies $\sigma^\varepsilon \ge 0$.\\
To prove (ii), we integrate \eqref{obs_adj} over $U$, to get
\begin{align}
\int_U (1+(\gamma^\varepsilon)')\sigma^\varepsilon\, dx 
= & \int_U \mbox{div}(D_pH \sigma^\varepsilon) \, dx +\varepsilon \int_U \Delta \sigma^\varepsilon \,dx+1\notag
\\
=& \int_{\partial U} (D_p H \cdot \nu) \sigma^\varepsilon \,dS 
+ \varepsilon \int_{\partial U} \dfrac{\partial \sigma^\varepsilon}{\partial \nu} \,dS+1 
=\varepsilon \int_{\partial U} \dfrac{\partial \sigma^\varepsilon}{\partial \nu} \,dS + 1 ,\notag
\end{align}
where we used the fact that $\sigma^\varepsilon=0$ on $\partial U$.
\end{proof}

Using the adjoint equation, we have the following new estimate.
\begin{Lemma}
\label{obs_lem3}
There exists $C>0$, independent of $\varepsilon > 0$, such that
\begin{equation}
\frac{1}{2} \int_U (1 + ( \gamma^{\varepsilon} )')|D u^\varepsilon|^2 \sigma^\varepsilon \,dx
+ \varepsilon \int_U |D^2 u^\varepsilon|^2 \sigma^\varepsilon \,dx\le C.
\label{obs4}
\end{equation}
\end{Lemma}
\begin{proof}
Multiplying \eqref{obs2} by $\sigma^\varepsilon$ and integrating by parts, using equation \eqref{obs_adj} we
get
\begin{align*}
&\frac{1}{2} \int_U (1 + ( \gamma^{\varepsilon} )')|D u^\varepsilon|^2 \sigma^\varepsilon\, dx
+ \varepsilon \int_U |D^2 u^\varepsilon|^2 \sigma^\varepsilon \,dx \\
&\hspace{.5cm}= - w^\varepsilon (x_0) 
- \int_{U} \left[ D_x H \cdot D u^{\varepsilon}  - ( \gamma^{\varepsilon} )' D \psi \cdot D u^{\varepsilon}
\right] \sigma^{\varepsilon} \, dx
- \varepsilon \int_{\partial U} w^{\varepsilon} \frac{\partial \sigma^{\varepsilon}}{\partial \nu} \, d S.
\end{align*}
Thanks to Lemma \ref{obs_lem2} and Proposition \ref{obs1} 
(which, in particular, implies 
$\| D_x H (\cdot , D u^{\varepsilon} (\cdot)) \|_{L^{\infty} (U)} \leq C$)
the conclusion follows.
\end{proof}

\noindent
Relation (\ref{obs4}) shows that we have a good control of the Hessian $D^2 u^\varepsilon$ 
in the support of $\sigma^\varepsilon$.

We finally have the following result, which immediately implies Theorem \ref{obs_speed}.
\begin{Lemma}
\label{obs_thm1}
There exists $C>0$, independent of $\varepsilon$, such that
\begin{equation}
\max_{x \in \overline{U}} \left| u^\varepsilon_{\varepsilon} (x) \right| \le
\dfrac{C}{\varepsilon^{1/2}}.
\notag
\end{equation}
\end{Lemma}
\begin{proof}
By standard elliptic estimates, the solution $u^\varepsilon$ is smooth in the parameter $\varepsilon$ for $\varepsilon>0$
(see \cite{E2, T1} for similar arguments).
Differentiating (\ref{obs_reg}) w.r.t. $\varepsilon$ we get
\begin{equation}
(1+(\gamma^\varepsilon)')u_\varepsilon^\varepsilon +D_pH\cdot Du_\varepsilon^\varepsilon 
+ \gamma_\varepsilon^\varepsilon = \varepsilon \Delta u_\varepsilon^\varepsilon + \Delta u^\varepsilon, 
\quad \quad \text{ in }U.
\label{obs5}
\end{equation} 
In addition, we have $u_\varepsilon^\varepsilon(x)=0$ for all $x\in \partial U$, 
since $u^\varepsilon(x)=0$ on $\partial U$ for every $\varepsilon$.
So, we may assume that there exists $x_2 \in U$ such that
$|u_\varepsilon^\varepsilon(x_2)|=\max_{x \in \overline{U}} |u_\varepsilon^\varepsilon (x)|$. 

Consider the adjoint equation \eqref{obs_adj}, and choose $x_0=x_2$.
Multiplying by $\sigma^\varepsilon$ both sides of (\ref{obs5}) and integrating by parts,
$$
u^\varepsilon_\varepsilon(x_2) = -  \int_U \gamma_\varepsilon^\varepsilon \sigma^\varepsilon \,dx+ \int_U \Delta u^\varepsilon \sigma^\varepsilon \,dx.
$$
Hence,
\begin{equation}
|u_\varepsilon^\varepsilon(x_2)| \le \int_U |\gamma_\varepsilon^\varepsilon |\sigma^\varepsilon \,dx+ \int_U |\Delta u^\varepsilon| \sigma^\varepsilon  \,dx.
\label{obsk}
\end{equation}
By Lemma \ref{obs_lem4},
\begin{equation}
|\gamma_\varepsilon^\varepsilon| 
= \left| -\dfrac{u^\varepsilon - \psi}{\varepsilon^2} \gamma' \left( \dfrac{u^\varepsilon-\psi}{\varepsilon}
\right) \right| 
= \left| \dfrac{u^\varepsilon - \psi}{\varepsilon} (\gamma^\varepsilon)'(u^\varepsilon-\psi) \right| \le C
(\gamma^\varepsilon)'.
\label{obs6}
\end{equation}
Hence, thanks to Lemma \ref{obs_lem2}
\begin{equation}
\int_U |\gamma_\varepsilon^\varepsilon| \sigma^\varepsilon \,dx \le C  \int_U (\gamma^\varepsilon)'
\sigma^\varepsilon \,dx\le C,
\label{uj}
\end{equation}
while using \eqref{obs4}
\begin{equation} \label{obs7} \begin{split}
 \int_U |\Delta u^\varepsilon| \sigma^\varepsilon \,dx 
  &\le \left( \int_U | \Delta u^\varepsilon|^2 \sigma^\varepsilon \,dx \right)^{1/2} 
 \left( \int_U\sigma^\varepsilon \,dx \right)^{1/2} \\
  &\le C\left( \int_U  |D^2 u^\varepsilon|^2 \sigma^\varepsilon \,dx \right)^{1/2} 
 \left( \int_U\sigma^\varepsilon \,dx \right)^{1/2}  \le \dfrac{C}{\varepsilon^{1/2}}.
\end{split} \end{equation}
Thus, by (\ref{obsk}), \eqref{uj} and (\ref{obs7})
\begin{equation}
 |u_\varepsilon^\varepsilon(x_2)| \le \dfrac{C}{ \varepsilon^{1/2}}, ~\mbox{for}~ \varepsilon<1.
\label{obs8}
\end{equation}
\end{proof}

\end{subsection}

\end{section}

\begin{section}{Weakly coupled systems of Hamilton--Jacobi equations} \label{weakcou}

We study now the model of monotone weakly coupled systems of Hamilton-Jacobi equations 
considered by Engler and Lenhart \cite{EL}, and by Ishii and Koike \cite{IK1}. 
For the sake of simplicity, we will just focus on the following system of two equations:
\begin{equation}
\left\{ \begin{aligned}
c_{11} u_1 +c_{12} u_2 + H_1(x,Du_1) &=0 \vspace{.05in}\\
c_{21} u_1 +c_{22} u_2 + H_2(x,D u_2) &=0 \\
\end{aligned} \right. 
\quad \mbox{in}~ U,
\label{w_coupled}
\end{equation}
with boundary conditions $u_1 = u_2 =0$ on $\partial U$.
The general case of more equations can be treated in a similar way.

We assume that the Hamiltonians 
$H_1, H_2:  \overline{U} \times \mathbb R^n \rightarrow \mathbb R$ are smooth satisfying
{the following hypotheses.}
%\vspace{.1cm}
\begin{itemize}
\item[(H\ref{weakcou}.1)] {There exists $\beta_1,\beta_2 >0$ such that for every $j=1,2$
$$
\lim_{|p| \to +\infty} \left( \beta_j |H_j(x,p)|^2+D_xH_j(x,p)\cdot p \right)
=\lim_{|p| \to+\infty} \frac{H_j(x,p)}{|p|} =+\infty \text{ uniformly in } x\in \overline{U}.
$$}
\end{itemize}
Following \cite{EL} and \cite{IK1}, we suppose further that
\begin{itemize}
\item[(H\ref{weakcou}.2)] $\displaystyle c_{12},\, c_{21} {\leq} 0$; 
\item[(H\ref{weakcou}.3)] there exists $\alpha>0$ such that
$\displaystyle c_{11}+c_{12}, ~c_{21}+c_{22} \geq \alpha>0$.
\end{itemize}
We observe that, as a consequence, we also have $c_{11}, c_{22}>0$.
Finally, we require that
\begin{itemize}
\item[(H\ref{weakcou}.4)] There exist $\Phi_1,  \Phi_2 \in C^2 (U) \cap C^1 (\overline U)$ 
with $\Phi_j =0$ on $\partial U$ ($j=1,2$), and such that
\begin{equation*} 
\left\{ \begin{aligned}
c_{11} \Phi_1 + c_{12} \Phi_2 +H_1 (x,D\Phi_1 ) 
< 0 \quad\mbox{in}~U, \vspace{.05in}\\
c_{22} \Phi_2 + c_{21} \Phi_1 +H_2 (x,D\Phi_2 ) 
< 0 \quad\mbox{in}~U. \vspace{.05in}
\end{aligned} \right.
\end{equation*}
\end{itemize}
Thanks to these conditions, the Maximum Principle can be applied and existence, 
comparison and uniqueness results hold true, as stated in \cite{EL}.

We consider now the following regularized system (here $\varepsilon >0$):
\begin{equation}
\left\{ \begin{aligned}
c_{11} u^\varepsilon_1 +c_{12} u^\varepsilon_2 + H_1(x,D u^\varepsilon_1) &=\varepsilon \Delta u^\varepsilon_1
\vspace{.05in}\\
c_{21}u^\varepsilon_1+c_{22}u^\varepsilon_2 + H_2(x,Du^\varepsilon_2) &=\varepsilon \Delta u^\varepsilon_2 \\
\end{aligned} \right. 
\quad \mbox{in}~ U,
\label{wc_reg}
\end{equation}
with boundary conditions $u^\varepsilon_1 = u^\varepsilon_2=0$ on $\partial U$.\\
Conditions (H\ref{weakcou}.1), (H\ref{weakcou}.2), and (H\ref{weakcou}.3) yield {existence} and uniqueness 
of the pair of solutions $(u^\varepsilon_1, u^\varepsilon_2)$ in \eqref{wc_reg}. 
Next lemma gives a uniform bound for the $C^1$ norm 
of the sequences $\{ u^{\varepsilon}_i \}$, $i=1,2$.
Its proof, which is very similar to that one of Proposition \ref{obs1},
is still presented for readers' convenience.
\begin{Lemma} \label{lemu}
There exists a positive constant $C$, independent of $\varepsilon$, such that
\begin{equation}
\|u^\varepsilon_i \|_{L^\infty}, \| Du^\varepsilon_i \|_{L^\infty} \le C,
\quad \text{for } i=1,2.
\notag
\end{equation}
\end{Lemma}
\begin{proof}

\textbf{Step I: Bound on $\|u^\varepsilon_j \|_{L^\infty}, j=1,2$.}

{First of all observe that $u^{\varepsilon}_1 = u^{\varepsilon}_2 = 0$ on $\partial U$ for every $\varepsilon$.
Thus, it will be sufficient to show that $u^{\varepsilon}_1$ and $u^{\varepsilon}_2$
are bounded in the interior of $U$.
Without loss of generality, we can assume that there exists $\overline{x} \in U$ such that
\begin{equation*}
\max_{\substack{ j=1,2 \\  x \in \overline{U} }} u^{\varepsilon}_j (x) = u^{\varepsilon}_1 (\overline{x}).
\end{equation*}
We have 
\begin{equation*}
\alpha u^{\varepsilon}_1 (\overline{x}) \leq c_{11} u^{\varepsilon}_1 (\overline{x}) + c_{12} u^{\varepsilon}_2 (\overline{x})
\leq - H_1 (\overline{x}, 0) \leq \max_{x \in \overline{U}} \left( - H_1 (x,0) \right),
\end{equation*}
where we used (H\ref{weakcou}.3) and equation \eqref{wc_reg}.
Analogously, if $\widehat{x} \in U$ is such that 
\begin{equation*}
\min_{\substack{ j=1,2 \\  x \in \overline{U} }} u^{\varepsilon}_j (x) = u^{\varepsilon}_1 (\widehat{x}), 
\end{equation*}
then
\begin{equation*}
u^{\varepsilon}_1 (\widehat{x}) \geq \frac{c_{11}}{c_{11} + c_{12} } u^{\varepsilon}_1 (\widehat{x}) 
+ \frac{c_{12}}{c_{11} + c_{12} } u^{\varepsilon}_2 (\widehat{x})
\geq - \frac{H_1 (\widehat{x}, 0)}{c_{11} + c_{12} } \geq \frac{1}{c_{11} + c_{12} } \min_{x \in \overline{U}} \left( - H_1 (x,0) \right).
\end{equation*}}
Concerning the bounds on the gradients, we will argue as in the proof of Proposition \ref{obs1}.

\textbf{Step II: Bound on $\| D u^\varepsilon_j \|_{L^\infty (\partial U)}, j=1,2$.}

{We now show that 
\begin{equation*}
\max_{\substack{ j=1,2 \\  x \in \partial U }} | D u^{\varepsilon}_j (x) | \leq C,
\end{equation*}
for some constant $C$ independent of $\varepsilon$.
As it was done in Section \ref{sectobs}, we are going to construct appropriate barriers.
For $\varepsilon$ small enough, assumption (H\ref{weakcou}.4) implies that
\begin{equation*} 
\left\{ \begin{aligned}
c_{11} \Phi_1 + c_{12} \Phi_2 +H_1 (x,D\Phi_1 ) 
< \varepsilon \Delta \Phi_1 \quad\mbox{in}~U, \vspace{.05in}\\
c_{22} \Phi_2 + c_{21} \Phi_1 +H_2 (x,D\Phi_2 ) 
< \varepsilon \Delta \Phi_2 \quad\mbox{in}~U, \vspace{.05in}
\end{aligned} \right.
\end{equation*}
and $\Phi_1 = \Phi_2 =0$ on $\partial U$. 
Therefore, $(\Phi_1,\Phi_2)$ is a sub-solution of \eqref{wc_reg}. 
By the comparison principle, $u^\varepsilon_j \ge \Phi_j$ in~$U$, $j=1,2$.
Let $d(x)$, $\delta$, and $U_{\delta}$ be as in the proof of Proposition \ref{obs1}. 
For $ \mu > 0 $ large enough, the uniform bounds on $\|u^\varepsilon_1\|_{L^\infty}$ and $\|u^\varepsilon_2\|_{L^\infty}$ yield
$v := \mu  d \ge u^\varepsilon_j$ on $\partial U_\delta $, $j=1,2$, so that
\begin{equation*} 
\left\{ \begin{aligned}
(c_{11} + c_{12} ) v +H_1(x,Dv) - \varepsilon \Delta v 
\ge H_1(x,\mu Dd) - \mu C \quad\mbox{in}~U, \vspace{.05in}\\
(c_{21} + c_{22} ) v +H_2(x,Dv) - \varepsilon \Delta v 
\ge H_2(x,\mu Dd) - \mu C \quad\mbox{in}~U. \vspace{.05in}\\
\end{aligned} \right.
\end{equation*}
Now, we have $\Phi_j= u^\varepsilon_j= v = 0$ on $\partial U$.
Also, thanks to assumption (H\ref{weakcou}.1), for $\mu >0$ large enough
\begin{equation*} 
\left\{ \begin{aligned}
(c_{11} + c_{12} ) v +H_1(x,Dv) - \varepsilon \Delta v 
 \ge 0 \quad\mbox{in}~U, \vspace{.05in}\\
(c_{21} + c_{22} ) v +H_2(x,Dv) - \varepsilon \Delta v 
\ge 0 \quad\mbox{in}~U, \vspace{.05in}\\
\end{aligned} \right.
\end{equation*}
that is, the pair $(v,v)$ is a super-solution for the system \eqref{wc_reg}.
Thus, the comparison principle gives us that $\Phi_j \le u^\varepsilon_j \le v$ in $ U_\delta$, $j=1,2$.
Then, from the fact that $\Phi_j= u^\varepsilon_j= v = 0$ on $\partial U$ we get 
$$
\dfrac{\partial v}{\partial \nu}(x) \le \dfrac{\partial u^\varepsilon_j}{\partial \nu}(x) \le \dfrac{\partial \Phi_j}{\partial \nu}(x),
\quad \text{for } x\in \partial U, \quad j=1,2.
$$
Hence, we obtain $\| Du^{\varepsilon}_j \|_{L^{\infty}(\partial U)} \leq C$, $j=1,2$.}

\textbf{Step III: Conclusion.}

{Setting $w^\varepsilon_j=\dfrac{|Du^\varepsilon_j|^2}{2}$, $j=1,2$, by a direct computation we have that
\begin{equation} \label{weqn}
\left\{ \begin{aligned}
2 c_{11} w^\varepsilon_1 + D_p H_1\cdot Dw^\varepsilon_1 + c_{12} Du^\varepsilon_1 \cdot Du^\varepsilon_2
+D_x H_1\cdot Du^\varepsilon_1 = \varepsilon \Delta w^\varepsilon_1 - \varepsilon |D^2 u^\varepsilon_1|^2, \vspace{.05in} \\
2 c_{22} w^\varepsilon_2 + D_p H_2\cdot D w^\varepsilon_2 + c_{21} Du^\varepsilon_1 \cdot D u^\varepsilon_2
+D_x H_2\cdot Du^\varepsilon_2 = \varepsilon \Delta w^\varepsilon_2 - \varepsilon |D^2 u^\varepsilon_2 |^2.
\end{aligned} \right.
\end{equation}
Assume now that there exists $\widehat{x} \in U$ such that 
\begin{equation*}
\max_{\substack{ j=1,2 \\  x \in \overline{U} }} w^{\varepsilon}_j (x) = w^{\varepsilon}_1 (\widehat{x}). 
\end{equation*}
Then, we have 
\begin{equation*} \begin{split}
\varepsilon |D^2 u^\varepsilon_1|^2 (\widehat{x})
&= \varepsilon \Delta w^\varepsilon_1  (\widehat{x})  - 2 c_{11} w^\varepsilon_1 (\widehat{x}) 
- c_{12} Du^\varepsilon_1 (\widehat{x}) \cdot Du^\varepsilon_2 (\widehat{x})
- D_x H_1\cdot Du^\varepsilon_1 (\widehat{x}) \\
&\leq - 2 (  c_{11} + c_{12} ) w^{\varepsilon}_1 (\widehat{x}) - D_x H_1\cdot Du^\varepsilon_1 (\widehat{x})
\leq - D_x H_1\cdot Du^\varepsilon_1 (\widehat{x}).
\end{split} \end{equation*}
Now, arguing as in the proof of Proposition \ref{obs1}, for $\varepsilon$ sufficiently small
\begin{equation*} \begin{split}
\varepsilon |D^2 u^\varepsilon_1 (\widehat{x}) |^2 
&\geq 2 \beta_1 \varepsilon^2 | \Delta u^\varepsilon_1 (\widehat{x})|^2
=2 \beta_1 \left[  c_{11} u^\varepsilon_1 (\widehat{x}) +c_{12} u^\varepsilon_2 (\widehat{x}) + H_1( \widehat{x} ,D u^\varepsilon_1(\widehat{x}))  \right]^2 \\
&\geq \beta_1 | H_1( \widehat{x} ,D u^\varepsilon_1(\widehat{x})) |^2 - C.
\end{split} \end{equation*}
Collecting the last two relations we have 
\begin{equation*}
 \beta_1 | H_1( \widehat{x} ,D u^\varepsilon_1(\widehat{x})) |^2
 + D_x H_1( \widehat{x} ,D u^\varepsilon_1(\widehat{x})) \cdot Du^\varepsilon_1 (\widehat{x}) \leq C.
\end{equation*}
Recalling condition (H\ref{weakcou}.1) the conclusion follows.}
\end{proof}
\noindent
{\bf Adjoint method.} At this point, we introduce the 
adjoint of the linearization of system \eqref{wc_reg}.
Let us emphasize that the adjoint equations
we introduce
form a system of weakly coupled type, which is very natural
in this setting, and create a systematic way to 
the study of weakly coupled systems.
The linearized operator corresponding to \eqref{wc_reg} is
\begin{equation*} 
L^{\varepsilon} (z_1,z_2) : =
\left\{ \begin{aligned}
D_p H_1 (x,Du_1^\varepsilon) \cdot D z_1 
+ c_{11} z_1 + c_{12} z_2- \varepsilon \Delta z_1 , \vspace{.05in}\\
D_p H_2 (x,Du_2^\varepsilon) \cdot D z_2 
+ c_{22} z_2  + c_{21} z_1- \varepsilon \Delta z_2.
\end{aligned} \right.
\end{equation*}
Let us now identify the adjoint operator $(L^{\varepsilon})^*$.
For every $\nu^1,\nu^2 \in C^{\infty}_c (U)$ we have
\begin{equation*} \begin{split}
&\langle (L^{\varepsilon})^* (\nu^1, \nu^2), (z_1, z_2) \rangle 
:= \langle (\nu^1, \nu^2), L^{\varepsilon} (z_1, z_2) \rangle \\
&= \langle \, \nu^1, \left[ L^{\varepsilon} (z_1, z_2) \right]_1 \, \rangle
+ \langle \, \nu^2, \left[ L^{\varepsilon} (z_1, z_2) \right]_2 \, \rangle \\
&= \int_U \left[ D_p H_1 (x,Du_1^\varepsilon) \cdot D z_1 
+ c_{11} z_1 + c_{12} z_2- \varepsilon \Delta z_1 \right] \, \nu^1 \, dx \\
&+ \int_U \left[ D_p H_2 (x,Du_2^\varepsilon) \cdot D z_2 
+ c_{22} z_2  + c_{21} z_1- \varepsilon \Delta z_2 \right] \, \nu^2 \, dx \\
&= \int_U \left[ - \mbox{div}(D_pH_1 \nu^1 ) + c_{11} \nu^1 + c_{21}  \nu^2
- \varepsilon \Delta \nu^1 \right] \, z_1 \, dx \\
&+ \int_U \left[ - \mbox{div}(D_pH_2 \nu^2 )  + c_{22} \nu^2 + c_{12} \nu^1
- \varepsilon \Delta \nu^2 \right] \, z_2 \, dx.
\end{split} \end{equation*}
Then, the adjoint equations are:
\begin{equation} \label{rfd0}
\left\{ \begin{aligned}
- \mbox{div}(D_pH_1 \sigma^{1,\varepsilon}) + c_{11} \sigma^{1,\varepsilon} + c_{21} \sigma^{2,\varepsilon}
&= \varepsilon \Delta
\sigma^{1,\varepsilon}+ (2 - i) \delta_{x_0} \quad &\mbox{in}~ U,\vspace{.05in}\\
- \mbox{div}(D_pH_2 \sigma^{2,\varepsilon}) + c_{22} \sigma^{2,\varepsilon} + c_{12} \sigma^{1,\varepsilon}
&= \varepsilon \Delta
\sigma^{2,\varepsilon}+ (i - 1) \delta_{x_0} \quad &\mbox{in}~ U,
\end{aligned} \right.
\end{equation}
with boundary conditions
\begin{equation*}
\left\{ \begin{aligned}
\sigma^{1,\varepsilon}&=0 \qquad \quad&\mbox{on}~\partial U, \vspace{.05in}\\
\sigma^{2,\varepsilon}&=0 \qquad \quad&\mbox{on}~\partial U ,\\
\end{aligned} \right.
\end{equation*}
where $i \in \{  1, 2 \} $ and $x_0 \in U$ will be chosen later.
{Observe that, once $x_0$ is given, the choice $i=1$ ($i=2$) corresponds to an adjoint system
of two equations
where  a Dirac delta measure concentrated at $x_0$ appears only on the right-hand side 
of the first (second) equation.}
Existence and uniqueness of $\sigma^{1,\varepsilon}$ and $\sigma^{2,\varepsilon}$ 
follow by Fredholm alternative, by arguing as in Section \ref{sectobs}, 
and we have $\sigma^{1,\varepsilon}, \sigma^{2,\varepsilon} \in C^{\infty} (U \setminus \{ x_0 \})$.
We study now further properties of $\sigma^{1,\varepsilon}$ and $\sigma^{2,\varepsilon}$.
\begin{Lemma}[Properties of $\sigma^{1,\varepsilon}, \sigma^{2,\varepsilon}$] \label{kiz}
Let $\nu$ be the outer unit normal to $\partial U$. Then
\begin{itemize}
\item[(i)] $\sigma^{j,\varepsilon} \ge 0$ on $\overline{U}$.  
In particular, $\dfrac{\partial \sigma^{j,\varepsilon}}{\partial \nu} \le 0$ on $\partial U \,\, (j=1,2)$.
\item[(ii)] The following equality holds: 
\begin{equation}
\sum_{j=1}^2 \left( \int_U  (c_{j1} + c_{j2}) \sigma^{j,\varepsilon} \,dx 
- \varepsilon \int_{\partial U} \dfrac{\partial \sigma^{j,\varepsilon}}{\partial \nu} \,dS \right) = 1 .
\notag
\end{equation}
In particular,
$$
\sum_{j=1}^2 \int_U (c_{j1} + c_{j2}) \sigma^{j,\varepsilon} \,dx  \le 1.
$$
\end{itemize}
\end{Lemma}

\begin{proof}
First of all, we consider the adjoint of equation \eqref{rfd0}:
\begin{equation} \label{adjadj50} 
\left\{ \begin{aligned}
D_p H_1 (x,Du_1^\varepsilon) \cdot D z_1 
+ c_{11} z_1 + c_{12} z_2- \varepsilon \Delta z_1 = f_1, \vspace{.05in}\\
D_p H_2 (x,Du_2^\varepsilon) \cdot D z_2 
+ c_{22} z_2  + c_{21} z_1- \varepsilon \Delta z_2 = f_2,
\end{aligned} \right.
\end{equation}
where $f_1,f_2 \in C (U)$, with boundary conditions $z_1=z_2 =0$ on $\partial U$.
Note that
\begin{equation} \label{posit0}
f_1,f_2 \geq 0 \Longrightarrow 
\min_{\substack{ j=1,2 \\  x \in \overline{U} }} z_j (x) \geq 0.
\end{equation}
Indeed, if the minimum is achieved for some $\overline{x} \in \partial U$, then clearly $z_1, z_2 \geq 0$.
Otherwise, assume 
\begin{equation*}
\min_{\substack{ j=1,2 \\  x \in \overline{U} }} z_j (x) = z_1 (\overline{x}), 
\end{equation*}
for some $\overline{x} \in U$.
Using condition (H\ref{weakcou}.2)
\begin{equation*}
(c_{11} + c_{12} ) z_1 (\overline{x}) \geq c_{11} z_1 (\overline{x}) + c_{12} z_2 (\overline{x})
= \varepsilon \Delta z_1 (\overline{x}) + f_1 (\overline{x}) \geq 0.
\end{equation*}
Thanks to (H\ref{weakcou}.3), \eqref{posit0} follows.

Let us now multiply \eqref{rfd0}$_1$ and \eqref{rfd0}$_2$ by the solutions $z_1$ and $z_2$ of \eqref{adjadj50}.
Adding up the relations obtained we have
\begin{equation*}
\int_U f_1 \sigma^{1,\varepsilon} \,  dx + \int_U f_2 \sigma^{2,\varepsilon} \,  dx
= (2 - i) z_1 (x_0) + (1 - i) z_2 (x_0).
\end{equation*}
Thanks to \eqref{posit0}, from last relation  we conclude that\begin{equation*} 
\int_U f_1 \sigma^{1,\varepsilon} \,  dx + \int_U f_2 \sigma^{2,\varepsilon} \,  dx \geq 0,
\quad \text{ for every }f_1,f_2 \geq 0,
\end{equation*}
and this implies that $\sigma^{1,\varepsilon}, \sigma^{2,\varepsilon} \geq 0$.
To prove (ii), it is sufficient to integrate equations \eqref{rfd0}$_1$ and \eqref{rfd0}$_2$ over $U$, 
and to add up the two relations obtained.
\end{proof}
{The proof of the next lemma can be obtained by arguing as in the proof of Lemma \ref{obs_lem3}.}
\begin{Lemma}
\label{wc_lem1}
There exists a constant $C>0$, independent of $\varepsilon$, such that
\begin{equation*}
\varepsilon  \int_U |D^2 u^\varepsilon_1 |^2 \sigma^{1,\varepsilon}\, dx
+ \varepsilon \int_U |D^2 u^\varepsilon_2 |^2 \sigma^{2,\varepsilon} \,dx \le C. 
\end{equation*}
\end{Lemma}
We now give the last lemma needed to estimate the speed of convergence.
Here we use the notation $u^{\varepsilon}_{j, \varepsilon} (x) := \partial u^{\varepsilon}_j (x) / \partial \varepsilon$, $j=1,2$.
\begin{Lemma}
\label{wc_lem2}
There exists a constant $C>0$, independent of $\varepsilon$, such that
\begin{equation}
\max_{\substack{ j=1,2 \\  x \in \overline{U} }} | u^{\varepsilon}_{j, \varepsilon} (x) | \leq \dfrac{C}{\varepsilon^{1/2}}.
\notag
\end{equation}
\end{Lemma}
\begin{proof}
Differentiating (\ref{wc_reg}) w.r.t $\varepsilon$ we obtain the system
\begin{equation} \label{juh}
\left\{ \begin{aligned}
c_{11}u^\varepsilon_{1,\varepsilon} +c_{12}u^\varepsilon_{2,\varepsilon} + D_p H_1\cdot Du^\varepsilon_{1,\varepsilon}
&=\varepsilon \Delta u^\varepsilon_{1,\varepsilon} + \Delta u^\varepsilon_1, \vspace{.05in}\\
c_{21}u^\varepsilon_{1,\varepsilon} + c_{22} u^\varepsilon_{2,\varepsilon} + D_p H_2\cdot Du^\varepsilon_{2,\varepsilon}
&=\varepsilon \Delta u^\varepsilon_{2,\varepsilon} +\Delta u^\varepsilon_2. \\
\end{aligned} \right. 
\end{equation}
Since $u^{\varepsilon}_{1,\varepsilon} =u^{\varepsilon}_{2,\varepsilon} = 0$ on $\partial U$,
we have 
\begin{equation*}
\max_{x \in \partial U} u^{\varepsilon}_{1,\varepsilon} (x)
= \max_{x \in \partial U} u^{\varepsilon}_{2,\varepsilon} (x) = 0.
\end{equation*}
Assume now that there exists $\widehat{x} \in U$ such that 
\begin{equation*}
\max_{\substack{ j=1,2 \\  x \in \overline{U} }} | u^{\varepsilon}_{j, \varepsilon} (x) | = | u^{\varepsilon}_{1, \varepsilon} (\widehat{x}) |,
\end{equation*}
and let $\sigma^{1,\varepsilon}, \sigma^{2,\varepsilon}$ be the solutions of system \eqref{rfd0}
with $i=1$ and $x_0 = \widehat{x}$.

Multiplying equations \eqref{juh}$_1$ and \eqref{juh}$_2$ by  
$\sigma^{1,\varepsilon}$ and $\sigma^{2,\varepsilon}$ respectively and adding up, 
thanks to \eqref{rfd0} we obtain
\begin{equation*}
u^{\varepsilon}_{1, \varepsilon} (\widehat{x})
= \int_U \Delta u^\varepsilon_1 \, \sigma^{1,\varepsilon} \, dx 
+ \int_U \Delta u^\varepsilon_2 \, \sigma^{2,\varepsilon} \, dx.
\end{equation*}
Thanks to Lemma \ref{wc_lem1}, and repeating the chain of inequalities in \eqref{obs7} one can show that
\begin{equation*}
\left| \int_U \Delta u^\varepsilon_j \, \sigma^{j,\varepsilon} \, dx \right|
\leq \frac{C}{\varepsilon^{1/2}},  \qquad j=1,2, 
\end{equation*}
and from this the conclusion follows.
\end{proof}
We can now prove the following result on the speed of convergence.
\begin{Theorem}
\label{wc_thm1}
There exists $C>0$, independent of $\varepsilon$, such that
$$
\|u^\varepsilon_1 - u_1 \|_{L^\infty}, \|u^\varepsilon_2 - u_2 \|_{L^\infty} \le C \varepsilon^{1/2}.
$$
\end{Theorem}

\begin{proof}
The theorem is a direct consequence of Lemma \ref{wc_lem2}.
\end{proof}

\end{section}

\begin{section}{Cell problem for Weakly coupled {systems} of Hamilton--Jacobi equations} \label{secweak}
In this {section we study the following class of weakly coupled systems} of Hamilton--Jacobi equations:
\begin{equation}
\left\{ \begin{aligned}
c_1u_1 - c_1 u_2 + H_1(x,Du_1) &=\overline{H}_1 \vspace{.05in}\\
- c_2 u_1 +c_2 u_2 + H_2(x,Du_2) &=\overline{H}_2 \\
\end{aligned} \right. 
\quad \mbox{in}~ \mathbb T^n,
\qquad \qquad \overline{H}_1, \overline{H}_2 \in \mathbb{R},
\label{cell}
\end{equation}
which is the analog of  the cell problem for single equation introduced by Lions, Papanicolaou, and Varadhan \cite{LPV}.
 We will assume that $H_1, H_2\in C^\infty(\mathbb T^n \times \mathbb R^n )$, and
\begin{itemize}

\item[(H\ref{secweak}.1)] {there exist $\omega_1, \omega_2>0$ such that for every $j=1,2$
$$
\lim_{|p| \to +\infty} \left[ \omega_j |H_j(x,p)|^2+D_x H_j (x,p)\cdot p - 16 n \omega_j c_j^2 |p|^2 \right] =+\infty \
\text{ uniformly in } x\in \mathbb T^n;
$$}

\item[(H\ref{secweak}.2)] $c_1, c_2>0$.
\end{itemize}
It is easy to see that the coefficients of $u_1 , u_2$ in this system do not satisfy 
the coupling assumptions of the previous section. 
Indeed, as it happens for the cell problem in the context of weak KAM theory,
there is no hope of a uniqueness result for (\ref{cell}). 

{
\begin{Remark}
The presence of the term $16 n \omega_j c_j^2$ in condition (H\ref{secweak}.1)
will be justified by later computations.
Nevertheless, we observe that (H\ref{secweak}.1) is weaker than (H\ref{weakcou}.1).
Indeed, if (H\ref{weakcou}.1) holds, then for every $\omega_j >\beta_j$
\begin{align*}
&\lim_{|p| \to +\infty} \left[ \omega_j |H_j(x,p)|^2+D_x H_j (x,p)\cdot p - 16 n \omega_j c_j^2 |p|^2 \right] \\
&= \lim_{|p| \to +\infty} \left[ \beta_j |H_j(x,p)|^2+D_x H_j (x,p) \cdot p  + (\omega_j - \beta_j) |H_j(x,p)|^2 -16  n \omega_j c_j^2 |p|^2 \right] = +\infty,
\end{align*}
uniformly in $x$, and hence (H\ref{secweak}.1) is satisfied. 
\end{Remark}}

To find the constants 
$\overline{H}_1, \overline{H}_2$
we use the same arguments as in \cite{T1}.  See also \cite{GomesSto, CGT1}.
First, for every $\varepsilon > 0$, let us consider the following regularized system:
\begin{equation}
\left\{ \begin{aligned}
(c_1+\varepsilon)u^\varepsilon_1 - c_1 u^\varepsilon_2 + H_1(x,Du^\varepsilon_1)
& = \varepsilon^2 \Delta u^\varepsilon_1
\vspace{.05in}\\
(c_2+\varepsilon) u^\varepsilon_2 - c_2 u^\varepsilon_1 + H_2(x,D u^\varepsilon_2 )
&=\varepsilon^2 \Delta u^\varepsilon_2 \\
\end{aligned} \right. 
\quad \mbox{in}~ \mathbb T^n.
\label{cell_reg}
\end{equation}
For every $\varepsilon>0$ fixed,
the coefficients of this new system satisfy the coupling assumptions
(H\ref{weakcou}.2) and (H\ref{weakcou}.3) of the previous section. 
Thus, (\ref{cell_reg}) admits a unique pair of smooth solutions $(u^\varepsilon_1 , u^\varepsilon_2)$.
In particular, this implies that $u^\varepsilon_1$ and $u^\varepsilon_2$ are $\mathbb T^n$-periodic. 

The following result gives some a priori estimates.
\begin{Theorem}
\label{cell_thm1}
There exists $C>0$, independent of $\varepsilon$, such that
$$
\| \varepsilon u^\varepsilon_1 \|_{L^\infty}, \| \varepsilon u^\varepsilon_2 \|_{L^\infty}, 
\|D u^\varepsilon_1 \|_{L^\infty}, \| D u^\varepsilon_2 \|_{L^\infty} \le C.
$$
\end{Theorem}
\begin{proof}
Our proof is based on the Maximum Principle.
Without loss of generality, we may assume that 
\begin{equation*}
\max_{\substack{ j=1,2 \\  x \in \mathbb T^n }} \left\{ \varepsilon u^{\varepsilon}_j (x) \right\} = \varepsilon u^{\varepsilon}_1 (x^\varepsilon_0){,}
\end{equation*}
for some $x^\varepsilon_0 \in \mathbb T^n$.
Applying the Maximum Principle to the first equation of (\ref{cell_reg}),
\begin{equation} \label{ol}
\varepsilon u^\varepsilon_1 (x^\varepsilon_0) \le (c_1+\varepsilon) u^\varepsilon_1 (x^\varepsilon_0)
-c_1u^\varepsilon_2 (x^\varepsilon_0) \le -H^1(x^\varepsilon_0,0) \le C,
\end{equation}
and this shows the existence of a bound from above for $\varepsilon u^{\varepsilon}_1$ and $\varepsilon u^{\varepsilon}_2$. 
Using a similar argument one can show that there is also a bound from below, so that
\begin{equation} \label{po}
\|\varepsilon u^\varepsilon_1 \|_{L^\infty}, \|\varepsilon u^\varepsilon_2 \|_{L^\infty} \le C.
\end{equation}
The previous inequality allows us to prove a bound for the difference 
$u^\varepsilon_1 (x^\varepsilon_0) - u^\varepsilon_2 (x^\varepsilon_0)$.
Indeed, thanks to {\eqref{ol} and} \eqref{po} we have
\begin{equation} \label{cell_inequ}
| u^\varepsilon_1 (x^\varepsilon_0) - u^\varepsilon_2 (x^\varepsilon_0) | 
= u^\varepsilon_1 (x^\varepsilon_0) - u^\varepsilon_2 (x^\varepsilon_0) 
\le - \frac{1}{c_1} H_1(x^\varepsilon_0,0) - \frac{\varepsilon}{c_1} u^\varepsilon_1 (x^\varepsilon_0) \le C.
\end{equation}
In order to find a bound for the gradients,
let us set $w^\varepsilon_j = \dfrac{|Du^\varepsilon_j|^2}{2}$, $j=1,2$.
Then, by a direct computation one can see that
\begin{equation*}
\left\{ \begin{aligned}
2(c_1+\varepsilon)w^\varepsilon_1 +D_p H_1 \cdot Dw^\varepsilon_1 - c_1 Du^\varepsilon_1 \cdot Du^\varepsilon_2 + D_x H_1 \cdot Du^\varepsilon_1 
&= \varepsilon^2 \Delta w^\varepsilon_1 -\varepsilon^2 |D^2 u^\varepsilon_1|^2 \vspace{.05in} \\
2(c_2+\varepsilon)w^\varepsilon_2 +D_p H_2\cdot Dw^\varepsilon_2 - c_2 Du^\varepsilon_1 \cdot Du^\varepsilon_2 + D_xH_2\cdot Du^\varepsilon_2 
&= \varepsilon^2 \Delta w^\varepsilon_2 -\varepsilon^2 |D^2 u^\varepsilon_2|^2
\end{aligned} \right. 
\quad \mbox{in}~ \mathbb T^n.
\end{equation*}
Without loss of generality, we may assume that there exists $x^\varepsilon_1 \in \mathbb T^n$ such that
\begin{equation*}
\max_{\substack{ j=1,2 \\  x \in \mathbb T^n }} \left\{ w^{\varepsilon}_j (x) \right\} = w^{\varepsilon}_1 (x^\varepsilon_1).
\end{equation*}
Then, by the Maximum Principle
\begin{equation}  \label{kjh} \begin{split}
\varepsilon^2 |D^2u^\varepsilon_1 (x^\varepsilon_1)|^2 
& \le -2(c_1+\varepsilon) w^\varepsilon_1 (x^\varepsilon_1) 
+ c_1Du^\varepsilon_1 (x^\varepsilon_1)\cdot Du^\varepsilon_2 (x^\varepsilon_1) - D_xH_1\cdot Du^\varepsilon_1 (x^\varepsilon_1) \\
& \le - D_xH_1\cdot Du^\varepsilon_1 (x^\varepsilon_1) .
\end{split} \end{equation}
Moreover, for $\varepsilon$ sufficiently small
\begin{equation} \label{gft}
\varepsilon^2 |D^2u^\varepsilon_1 (x^\varepsilon_1)|^2  \ge2 \omega_1 \varepsilon^4 (\Delta u^\varepsilon_1 (x^\varepsilon_1))^2  =
2 \omega_1 \left[ H_1(x^\varepsilon_1,Du^\varepsilon_1 (x^\varepsilon_1)) + (c_1+\varepsilon)u^\varepsilon_1 (x^\varepsilon_1)-c_1u^\varepsilon_2 (x^\varepsilon_1) \right]^2.
\end{equation}
Also, thanks to \eqref{po} and \eqref{cell_inequ}
\begin{align}
&|(c_1+\varepsilon)u^\varepsilon_1 (x^\varepsilon_1)- c_1 u^\varepsilon_2 (x^\varepsilon_1)| \notag\\
&\hspace{.2cm} \leq \varepsilon |u^\varepsilon_1 (x^\varepsilon_1)|+c_1 |u^\varepsilon_1 (x^\varepsilon_1)- u^\varepsilon_1 (x^\varepsilon_0)|
+c_1 |u^\varepsilon_2 (x^\varepsilon_1) - u^\varepsilon_2 (x^\varepsilon_0)| 
+c_1 |u^\varepsilon_1 (x^\varepsilon_0)-  u^\varepsilon_2(x^\varepsilon_0)| \notag\\
&\hspace{.2cm}\leq C+c_1|u^\varepsilon_1 (x^\varepsilon_1)- u^\varepsilon_1 (x^\varepsilon_0)|
+c_1 | u^\varepsilon_2 (x^\varepsilon_1) - u^\varepsilon_2 (x^\varepsilon_0)|\notag\\
&\hspace{.2cm}\leq C+ 2c_1 |Du^\varepsilon_1 (x^\varepsilon_1)| |x^\varepsilon_1-x^\varepsilon_0| 
\leq C+2c_1\sqrt{n}  |Du^\varepsilon_1 (x^\varepsilon_1)|, \notag
\end{align}
where we used the fact that the diameter of $\mathbb T^n$ is $\sqrt{n}$.
{Last relation, thanks to \eqref{gft} and Young's inequality, gives that
\begin{align}
\varepsilon^2 |D^2u^\varepsilon_1 (x^\varepsilon_1)|^2  
&\ge 2 \omega_1 \left[ H_1(x^\varepsilon_1,Du^\varepsilon_1 (x^\varepsilon_1)) + (c_1+\varepsilon)u^\varepsilon_1 (x^\varepsilon_1)-c_1u^\varepsilon_2 (x^\varepsilon_1) \right]^2 \notag\\
&\geq \omega_1 |H_1(x^\varepsilon_1,Du^\varepsilon_1 (x^\varepsilon_1)) |^2 
- 2 \omega_1 \left[  (c_1+\varepsilon)u^\varepsilon_1 (x^\varepsilon_1)-c_1u^\varepsilon_2 (x^\varepsilon_1) \right]^2  \notag\\
&\geq \omega_1 |H_1(x^\varepsilon_1,Du^\varepsilon_1 (x^\varepsilon_1)) |^2 
- 2 \omega_1 \left[  C+2c_1\sqrt{n}  |Du^\varepsilon_1 (x^\varepsilon_1)| \right]^2  \notag\\
&\geq \omega_1 |H_1(x^\varepsilon_1,Du^\varepsilon_1 (x^\varepsilon_1)) |^2
-C-16 n \omega_1 c_1^2 |Du^\varepsilon_1 (x^\varepsilon_1)|^2, \notag
\end{align}
Using last inequality and \eqref{kjh} we have
$$
\omega_1 |H_1(x_1,Du^\varepsilon_1 (x^\varepsilon_1)) |^2 +D_xH_1\cdot Du^\varepsilon_1 (x^\varepsilon_1)
- 16 n \omega_1 c_1^2 |Du^\varepsilon_1 (x^\varepsilon_1)|^2  \le C.
$$
Thanks to condition (H\ref{secweak}.1), we obtain the conclusion.}
\end{proof}

In the sequel, all the functions will be regarded as functions defined in the whole $\mathbb{R}^n$
and $\mathbb{Z}^n$-periodic.
Next lemma provides some a priori bounds on $u^\varepsilon_1$ and $u^\varepsilon_2$.
\begin{Lemma}
\label{cell_lem3}
There exists a constant $C>0$, independent of $\varepsilon$, such that
$$
|u^\varepsilon_1 (x) - u^\varepsilon_1 (y)| , | u^\varepsilon_2 (x) - u^\varepsilon_2 (y) |
, |u^\varepsilon_1 (x) - u^\varepsilon_2 (y)| \le C,\quad x,y \in \mathbb R^n.
$$
\end{Lemma}
\begin{proof}
The first two inequalities follow from the periodicity of 
$u^\varepsilon_1$ and $u^\varepsilon_2$, and from the fact that 
$Du^\varepsilon_1$ and $Du^\varepsilon_2$ are bounded.

Let us now show the last inequality.
{As in the previous proof, without loss of generality} we may assume that
there exists ${x_0^{\varepsilon}} \in \mathbb T^n$ such that
\begin{equation*} 
\max_{\substack{ j=1,2 \\  x \in \mathbb T^n }} \left\{ u^{\varepsilon}_j (x) \right\} 
= u^{\varepsilon}_1 (x^\varepsilon_0).
\end{equation*}
Combining the second inequality of the lemma with \eqref{ol},
$$
u^\varepsilon_1 (x) - u^\varepsilon_2 (y) 
\leq u^\varepsilon_1 (x^\varepsilon_0) - u^\varepsilon_2 (x^\varepsilon_0) 
+ u^\varepsilon_2 (x^\varepsilon_0) - u^\varepsilon_2 (y) \le C, \qquad x,y \in \mathbb R^n.
$$
The proof can be concluded by repeating the same argument for 
$\min_{\substack{ j=1,2 \\  x \in \mathbb T^n }} \left\{ u^{\varepsilon}_j (x) \right\}$.
\end{proof}

\begin{proof}[Proof of Theorem \ref{cell_const}]

{Thanks to Theorem \ref{cell_thm1} and Lemma \ref{cell_lem3}, 
\begin{equation} \label{Cbar}
\varepsilon u^\varepsilon_i \rightarrow \overline{C} %-\overline{H}, 
\quad \text{ uniformly in }\mathbb{T}^n,
\ \text{for } i=1,2,
\end{equation}
for some constant $\overline{C} \in \mathbb{R}$. %$\overline{H} \in \mathbb{R}$. 
Furthermore, still up to subsequences, 
\begin{equation} \label{convergences}
\left\{ \begin{aligned}
u^\varepsilon_1 - \min_{\mathbb T^n} u^\varepsilon_1 &\rightarrow u_1, \vspace{.05in}\\
u^\varepsilon_2 - \min_{\mathbb T^n} u^\varepsilon_2 &\rightarrow u_2,  \\
\end{aligned} \right. 
\quad \mbox{ and } \quad
\quad
\left\{ \begin{aligned}
-\varepsilon u^\varepsilon_1+
c_1( \min_{\mathbb T^n} u^\varepsilon_1- \min_{\mathbb T^n} u^\varepsilon_2) &\to \overline{H}_1, \vspace{.05in}\\
-\varepsilon u^\varepsilon_2+
c_2( \min_{\mathbb T^n} u^\varepsilon_2- \min_{\mathbb T^n} u^\varepsilon_1) &\to \overline{H}_2, \\
\end{aligned} \right. 
\end{equation}
uniformly in $\mathbb T^n$, for some functions $ {u_1, u_2} \in C(\mathbb{T}^n)$ 
and some constants $\overline{H}_1, \overline{H}_2 \in \mathbb{R}$.
From \eqref{convergences} it follows that the functions $(u_1,u_2)$ and the constants $(\overline{H} _1, \overline{H} _2)$
are such that (\ref{cell}) holds, in the viscosity sense.} 
\end{proof}

{\begin{Remark} \label{nouniq}
In general, $\overline{H}_1$ and $\overline{H}_2$ are not unique. 
Indeed, let $(u_1,u_2)$ be a viscosity solution of (\ref{cell}). 
Then, for every pair of constants $(C_1, C_2)$, the pair of functions
$(\widetilde{u}_1,\widetilde{u}_2)$ where
$\widetilde{u}_1 := u_1 +C_1$ and $\widetilde{u}_2 := u_2+C_2$ is still a viscosity solution
of (\ref{cell}), {with}
$$
\widetilde{H}_1 := \overline{H}_1 +c_1 (C_1-C_2), \quad \quad \quad 
~\widetilde{H}_2 := \overline{H}_2 +c_2(C_2-C_1),
$$
in place of $\overline{H}_1$ and $\overline{H}_2$, respectively.
Anyway, we have $c_2 \overline{H}_1 + c_1 \overline{H}_2
=c_2 \widetilde{H}_1 + c_1 \widetilde{H}_2$.
This suggests that, although $\overline{H}_1$ and $\overline{H}_2$ may vary, 
the expression $c_2 \overline{H}_1 + c_1 \overline{H}_2$ is unique. 
Theorem \ref{cell_const2} shows that this is the case.
\end{Remark}}

\begin{proof}[Proof of Theorem \ref{cell_const2}]
Without loss of generality, we may assume $c_1=c_2=1$.
Suppose, by contradiction, that there exist two pairs $(\lambda_1, \lambda_2) \in \mathbb R^2$ 
and $(\mu_1, \mu_2) \in \mathbb R^2$, 
and four functions $u_1, u_2, \widetilde{u}_1, \widetilde{u}_2\in C(\mathbb T^n)$ 
such that $\lambda_1+\lambda_2<\mu_1+\mu_2$ and 
\begin{equation}
\left\{ \begin{aligned}
u_1-u_2 + H_1(x,Du_1) &=\lambda_1 \vspace{.05in}\\
-u_1+u_2 + H_2(x,Du_2) &=\lambda_2 \\
\end{aligned} \right. 
\quad \mbox{in}~ \mathbb T^n,
\notag
\end{equation}
and
\begin{equation}
\left\{ \begin{aligned}
 \widetilde{u}_1 - \widetilde{u}_2 + H_1(x,D \widetilde{u}_1) &=\mu_1 \vspace{.05in}\\
 - \widetilde{u}_1 + \widetilde{u}_2 + H_2(x,D \widetilde{u}_2) &=\mu_2 \\
\end{aligned} \right. 
\quad \mbox{in}~ \mathbb T^n.
\notag
\end{equation}
By possibly substituting $u_1$ and $u_2$
with functions $\widehat{u}_1 := u_1 + C_1$ and $\widehat{u}_2 := u_2 + C_2$,
for suitable constants $C_1$ and $C_2$, we may always assume that 
$\lambda_1<\mu_1, \lambda_2<\mu_2$.

In the same way, by a further substitution  
$\overline{u}_1 := u_1 + C_3$, $\overline{u}_2 := u_2 + C_3$, with $C_3 > 0$ large enough,
we may assume that $u_1 > \widetilde{u}_1, u_2 > \widetilde{u}_2$.
Then, there exists $\varepsilon>0$ small enough such that
\begin{equation}
\left\{ \begin{aligned}
 (\varepsilon+1) u_1-u_2 + H_1(x,Du_1) & <  (\varepsilon+1) \widetilde{u}_1 - \widetilde{u}_2 + H_1(x,D\widetilde{u}_1) \vspace{.05in}\\
(\varepsilon+1) u_2 - u_1 + H_2 (x,Du_2) & < (\varepsilon+1)\widetilde{u}_2-\widetilde{u}_1+ H_2(x,D\widetilde{u}_2) \\
\end{aligned} \right. 
\quad \mbox{in}~ \mathbb T^n.
\notag
\end{equation}
Observe that the coefficients of the last system satisfy the coupling assumptions
(H\ref{weakcou}.2) and (H\ref{weakcou}.3).
Hence, applying the Comparison Principle in \cite{EL} and \cite{IK1}, 
we conclude that $u_1 < \widetilde{u}_1$ and $u_2 < \widetilde{u}_2$, which gives a contradiction.
\end{proof}

{\begin{Remark} \label{Hbar}
Multiplying the two convergences in the right in \eqref{convergences} by $c_2$ and $c_1$, 
respectively, one can see that 
$$
- \overline{C} = \overline{H}= \dfrac{\mu}{c_1+c_2}.
$$
Here, $\overline{C}$ is defined in \eqref{Cbar}, $\mu$ given by Theorem \ref{cell_const2},
and $\overline{H}$ is the unique constant  such that \eqref{falsj} has viscosity solutions.
We call $\overline{H}$ the \textit{effective Hamiltonian} of 
the cell problem for the weakly coupled system of Hamilton--Jacobi equations.
\end{Remark}}
The following is the main theorem of the section.
See also \cite{T1} for similar results.
\begin{Theorem}
\label{cell_thm2}
There exists a constant $C>0$, independent of $\varepsilon$, such that
$$
\|\varepsilon u^\varepsilon_i + \overline{H} \|_{L^\infty} \le C \varepsilon,
\quad \text{for } i=1,2.
$$
\end{Theorem}

{\bf Adjoint method:} Also in this case, we introduce the adjoint equations
associated to the linearization of the original problem.
We look for $\sigma^{1,\varepsilon}, \sigma^{2,\varepsilon}$ 
which are $\mathbb{T}^n$-periodic and such that 
\begin{equation}\label{cell_adj1}
\left\{ \begin{aligned}
- \mbox{div}(D_pH_1 \sigma^{1,\varepsilon}) + ( c_{1} + \varepsilon ) \sigma^{1,\varepsilon} - c_{2} \sigma^{2,\varepsilon} 
&= \varepsilon^2 \Delta \sigma^{1,\varepsilon}+ \varepsilon (2 - i) \delta_{x_0} \quad &\mbox{in}~ \mathbb T^n, \vspace{.05in}\\
- \mbox{div}(D_pH_2 \sigma^{2,\varepsilon}) + ( c_{2} + \varepsilon ) \sigma^{2,\varepsilon} - c_{1} \sigma^{1,\varepsilon}
&= \varepsilon^2 \Delta \sigma^{2,\varepsilon}+ \varepsilon (i - 1) \delta_{x_0} \quad &\mbox{in}~ \mathbb T^n,
\end{aligned} \right.
\end{equation}
where $i \in \{  1, 2 \} $ and $x_0 \in \mathbb T^n$ will be chosen later.
The argument used in Section \ref{sectobs} gives also in this case 
existence and uniqueness for $\sigma^{1,\varepsilon}$ and $\sigma^{2,\varepsilon}$. 
As before, we also have $\sigma^{1,\varepsilon}, \sigma^{2,\varepsilon} \in C^{\infty} (\mathbb T^n \setminus \{ x_0 \})$.
{The next two lemmas can be proven 
as Lemma \ref{kiz} and Lemma \ref{wc_lem1}, respectively.} 
\begin{Lemma}[Properties of $\sigma^{1,\varepsilon}, \sigma^{2,\varepsilon}$] \label{kir}
The functions $\sigma^{1,\varepsilon}, \sigma^{2,\varepsilon}$ satisfy the following:
\begin{itemize}
\item[(i)] $\sigma^{j,\varepsilon} \ge 0$ on $\mathbb{T}^n$   \,\, $(j=1,2)$;
\item[(ii)] Moreover, the following equality holds: 

$$
\sum_{j=1}^2 \int_{\mathbb{T}^n} \sigma^{j,\varepsilon} \,dx  = 1.
$$
\end{itemize}
\end{Lemma}

\begin{Lemma}
\label{cell_lem1}
There exists a constant $C>0$, independent of $\varepsilon$, such that
$$
\varepsilon^2\int_{\mathbb R^n} |D^2 u^\varepsilon_1|^2 \sigma^{1,\varepsilon} \,dx  \le C,
$$
$$
\varepsilon^2 \int_{\mathbb R^n} |D^2 u^\varepsilon_2|^2 \sigma^{2,\varepsilon} \,dx \le C.
$$
\end{Lemma}
Finally, next lemma allows us to prove Theorem \ref{cell_thm2}.
\begin{Lemma}
\label{cell_lem2}
There exists a constant $C>0$, independent of $\varepsilon$, such that
$$
\max_{\mathbb T^n} |(\varepsilon u^\varepsilon_1)_\varepsilon|, 
\quad \quad \max_{\mathbb T^n} |(\varepsilon u^\varepsilon_2)_\varepsilon| \le C.
$$
\end{Lemma}
\begin{proof}
Differentiating (\ref{cell_reg}) w.r.t. $\varepsilon$,
\begin{equation}
\left\{ \begin{aligned}
D_pH_1\cdot Du^\varepsilon_{1,\varepsilon} + (c_1+\varepsilon) u^\varepsilon_{1,\varepsilon} 
+ u^\varepsilon_1 - c_1 u^\varepsilon_{2,\varepsilon} 
&=\varepsilon^2 \Delta u^\varepsilon_{1,\varepsilon} 
+ 2 \varepsilon \Delta u^\varepsilon_1, \vspace{.05in}\\
D_pH_2\cdot Du^\varepsilon_{2,\varepsilon} + (c_2+\varepsilon) u^\varepsilon_{2,\varepsilon} 
+ u^\varepsilon_2 - c_2 u^\varepsilon_{1,\varepsilon}
&= \varepsilon^2 \Delta u^\varepsilon_{2,\varepsilon} + 2\varepsilon \Delta u^\varepsilon_2, \\
\end{aligned} \right. 
\notag
\end{equation}
where we set $u^\varepsilon_{j,\varepsilon} := \partial u^\varepsilon_{j} / \partial \varepsilon$, $j=1,2$.
Without loss of generality, we may assume that there exists $x_2\in {\mathbb T^n}$ such that
\begin{equation*}
%\label{cell1}
\max_{\substack{ j=1,2 \\  x \in \mathbb T^n }} ( \varepsilon u^{\varepsilon}_j (x))_{\varepsilon}
= \max_{\substack{ j=1,2 \\  x \in \mathbb T^n }} 
\left\{ \varepsilon  u^{\varepsilon}_{j, \varepsilon} (x) + u^{\varepsilon}_{j} (x) \right\}
= \varepsilon u^\varepsilon_{1,\varepsilon} (x_2)
+ u^{\varepsilon}_{1} (x_2).
\end{equation*} 
Choosing $x_0=x_2$ in the adjoint equation (\ref{cell_adj1}),  
and repeating the {steps in the proof of} Theorem~\ref{obs_thm1}, we get
\begin{equation}
\label{cell2} \begin{split}
&\varepsilon u^\varepsilon_{1,\varepsilon} (x_2)
+\int_{\mathbb T^n} u^\varepsilon_1 \sigma^{1,\varepsilon}\,dx
+ \int_{\mathbb T^n} u^\varepsilon_2 \sigma^{2,\varepsilon}\,dx \\
&\hspace{.2cm}\le 2\varepsilon \int_{\mathbb T^n} |\Delta
u^\varepsilon_1| \sigma^{1,\varepsilon}\,dx 
+ 2\varepsilon \int_{\mathbb T^n} |\Delta
u^\varepsilon_2| \sigma^{2,\varepsilon}\,dx\le C, 
\end{split} \end{equation}
where the latter inequality follows by repeating the chain of inequalities in \eqref{obs7}
and thanks to Lemma \ref{cell_lem1}.
Using Lemma \ref{cell_lem3} and property (ii) of Lemma \ref{kir} we have
\begin{equation*} \begin{split}
&\left| \int_{\mathbb T^n} u^\varepsilon_1 (x) \sigma^{1,\varepsilon}\,dx
+ \int_{\mathbb T^n} u^\varepsilon_2 (x) \sigma^{2,\varepsilon}\,dx
- u^{\varepsilon}_{1} (x_2) \right| \\
&\hspace{.5cm}= \left| \int_{\mathbb T^n} ( u^\varepsilon_1 (x) - u^{\varepsilon}_{1} (x_2) ) \sigma^{1,\varepsilon}\,dx
+ \int_{\mathbb T^n} ( u^\varepsilon_2 (x) - u^{\varepsilon}_{1} (x_2) ) \sigma^{2,\varepsilon}\,dx \right| \leq C.
\end{split} \end{equation*}
In view of the previous inequality, \eqref{cell2} becomes
\begin{equation*}
\varepsilon u^\varepsilon_{1,\varepsilon} (x_2)
+ u^{\varepsilon}_{1} (x_2) \leq C,
\end{equation*}
thus giving the bound from above.
The same argument, applied to $\min_{\substack{ j=1,2 \\  x \in \mathbb T^n }} ( \varepsilon u^{\varepsilon}_j (x))_{\varepsilon}$,
allows to prove the bound from below.
\end{proof}

\begin{proof}[Proof of Theorem \ref{cell_thm2}]

The theorem immediately follows by using Lemma \ref{cell_lem2}. 

\end{proof}

\begin{Remark}

(i). In order to achieve existence and uniqueness of the effective Hamiltonian $\overline{H}$
one can require either  (H4.1) or the usual coercive assumption (i.e. 
$H_1, H_2$ are uniformly coercive in $p$). 
Indeed one can consider the regularized system
\begin{equation}
\left\{ \begin{aligned}
(c_1+\varepsilon)u^\varepsilon_1 - c_1 u^\varepsilon_2 + H_1(x,Du^\varepsilon_1)
& = 0
\vspace{.05in}\\
(c_2+\varepsilon) u^\varepsilon_2 - c_2 u^\varepsilon_1 + H_2(x,D u^\varepsilon_2 )
&=0 \\
\end{aligned} \right. 
\quad \mbox{in}~ \mathbb T^n,
\label{cell_rem}
\end{equation}
and derive the results similarly to what we did above by using the coercivity of $H_1, H_2$.
We require {\rm (H4.1)} in order to get 
the speed of convergence as in Theorem \ref{cell_thm2}. \\
(ii). By using the same arguments, we can show that for any $P \in \mathbb R^n$,
there exist a pair of constants $(\overline{H}_1(P), \overline{H}_2(P))$ such that the system
\begin{equation*}
\left\{ \begin{aligned}
c_1u_1 - c_1 u_2 + H_1(x,P+Du_1)
& = \overline{H}_1(P)
\vspace{.05in}\\
c_2 u_2 - c_2 u_1 + H_2(x,P+D u_2 )
&=\overline{H}_2(P) \\
\end{aligned} \right. 
\quad \mbox{in}~ \mathbb T^n,
\end{equation*}
admits a solution $(u_1(\cdot,P),u_2(\cdot,P)) \in C(\mathbb T^n)^2$.
Moreover $\overline{H}(P)$, the effective Hamiltonian, is unique and
$$
\overline{H}(P)=\dfrac{c_2\overline{H}_1(P)+c_1\overline{H}_2(P)}{c_1+c_2}.
$$
\end{Remark}

\end{section}

\begin{section}{weakly coupled systems of obstacle type} \label{kilo}

In this last section we apply the Adjoint Method 
to weakly coupled systems of obstacle type.
Let $H_1, H_2: \overline{U} \times \mathbb{R}^n \to \mathbb{R}$ be smooth Hamiltonians, 
and let  $\psi_1, \psi_2 : \overline{U} \to \mathbb{R}$ be smooth functions describing the obstacles.
We assume that there exists $\alpha > 0$ such that
\begin{equation} \label{strpos}
 \psi_1, \, \, \psi_2 \geq \alpha \quad \text{ in }\overline{U},
\end{equation}
and consider the system
\begin{equation} \label{ki}
\left\{ \begin{aligned}
\max\{ u_1 - u_2 - \psi_1 , u_1+H_1(x,Du_1) \}&=0 \quad\mbox{in}~ U, \vspace{.05in}\\
\max\{  u_2 - u_1 - \psi_2 , u_2 +H_2(x, D u_2 ) \}&=0 \quad\mbox{in}~ U, \\
\end{aligned} \right.
\end{equation}
with boundary conditions $u_1 \mid_{\partial U}= u_2 \mid_{\partial U}= 0$.
We observe that \eqref{strpos} guarantees the compatibility of the boundary conditions, 
since $ \psi_1, \psi_2 > 0$ on $\partial U$.

Although the two equations in \eqref{ki} are coupled just through the difference $u_1 - u_2$, 
this problem turns out to be more difficult that the correspondent scalar equation \eqref{obs_eqn} studied in Section \ref{sectobs}.
For this reason, the hypotheses we require now are stronger.
We assume that

\begin{itemize}

\item[(H\ref{kilo}.1)] $H_j (x, \cdot)$ is convex for every $x \in \overline{U}$,  $j=1,2$.

\item[(H\ref{kilo}.2)] Superlinearity in $p$:
\begin{equation*}
\lim_{|p| \to \infty} \frac{H_j (x,p)}{|p|} = + \infty \quad \quad \text{ uniformly in }x, \quad j=1,2.
\end{equation*}

 \item[(H\ref{kilo}.3)] $|D_x H_j (x,p)| \leq C$ for each $(x,p) \in \overline{U} \times \mathbb{R}^n$, $j=1,2$.

 \item[(H\ref{kilo}.4)] There exist $\Phi_1,  \Phi_2 \in C^2 (U) \cap C^1 (\overline U)$ 
with $\Phi_j =0$ on $\partial U$ ($j=1,2$),  $- \psi_2 \leq \Phi_1 - \Phi_2 \leq \psi_1$, and such that
 $$
 \Phi_j +H_j (x,D\Phi_j) < 0\quad \mbox{in}~ \overline{U} \quad (j=1,2).
 $$
\end{itemize}
Let $\varepsilon > 0$ and let $\gamma^{\varepsilon}: \mathbb{R} \to [0, + \infty)$
be the function defined by \eqref{gammaep}.
We make in this section the additional assumption that $\gamma$ is convex.
We approximate \eqref{ki} by the following system
\begin{equation} \label{tg}
\left\{ \begin{aligned}
u^\varepsilon_1 +H_1 (x,Du^\varepsilon_1 ) 
+\gamma^\varepsilon(u^\varepsilon_1 - u^{\varepsilon}_2 - \psi_1 ) 
= \varepsilon \Delta u^\varepsilon_1 \quad\mbox{in}~U, \vspace{.05in}\\
u_2^\varepsilon+H_2 (x,D u_2^\varepsilon) 
+\gamma^\varepsilon(u_2^\varepsilon- u_1^{\varepsilon} - \psi_2 ) 
= \varepsilon \Delta u_2^\varepsilon \quad\mbox{in}~U. \vspace{.05in}
\end{aligned} \right.
\end{equation}
We are now ready to state the main result of the section.
\begin{Theorem}
\label{obs_speedfin}
There exists a positive constant $C$, independent of
$\varepsilon$, such that
\begin{equation}
 \|u^\varepsilon_i - u_i\|_{L^\infty}  \le C \varepsilon^{1/2},
\quad \text{for } i=1,2.
\notag
\end{equation}
\end{Theorem}
In order to prove the theorem we need several lemmas.
In the sequel, we shall use the notation
\begin{equation*}
\theta^\varepsilon_1:= u^\varepsilon_1- u_2^{\varepsilon} - \psi_1 , 
\quad \quad  
\theta^\varepsilon_2:= u_2^\varepsilon- u_1^{\varepsilon} - \psi_2 .
\end{equation*}
The linearized operator corresponding to \eqref{tg} is
\begin{equation*} 
L^{\varepsilon} (z_1,z_2) : =
\left\{ \begin{aligned}
z_1 + D_p H_1 (x,Du_1^\varepsilon) \cdot D z_1 
+  (\gamma^\varepsilon)' \mid_{\theta^\varepsilon_1} (z_1 - z_2)
- \varepsilon \Delta z_1 , \vspace{.05in}\\
z_2 + D_p H_2 (x,Du_2^\varepsilon) \cdot D z_2 
+  (\gamma^\varepsilon)' \mid_{\theta^\varepsilon_2} (z_2 - z_1)
- \varepsilon \Delta z_2.
\end{aligned} \right.
\end{equation*}
Then, the adjoint equations are:
\begin{equation} \label{rfd}
\left\{ \begin{aligned}
(1 + (\gamma^\varepsilon)' \mid_{\theta^\varepsilon_1})
\sigma^{1,\varepsilon}- \mbox{div}(D_pH_1 \sigma^{1,\varepsilon}) 
- (\gamma^\varepsilon)' \mid_{\theta^\varepsilon_2} \sigma^{2,\varepsilon}
&= \varepsilon \Delta
\sigma^{1,\varepsilon}+ (2 - i) \delta_{x_0} \quad &\mbox{in}~ U,\vspace{.05in}\\
(1 + (\gamma^\varepsilon)' \mid_{\theta^\varepsilon_2})
\sigma^{2,\varepsilon}- \mbox{div}(D_pH_2 \sigma^{2,\varepsilon}) 
- (\gamma^\varepsilon)' \mid_{\theta^\varepsilon_1} \sigma^{1,\varepsilon}
&= \varepsilon \Delta
\sigma^{2,\varepsilon}+ (i - 1) \delta_{x_0} \quad &\mbox{in}~ U,
\end{aligned} \right.
\end{equation}
with boundary conditions
\begin{equation*}
\left\{ \begin{aligned}
\sigma^{1,\varepsilon}&=0 \qquad \quad&\mbox{on}~\partial U, \vspace{.05in}\\
\sigma^{2,\varepsilon}&=0 \qquad \quad&\mbox{on}~\partial U ,\\
\end{aligned} \right.
\end{equation*}
where $i \in \{  1, 2 \} $ and $x_0 \in U$ will be chosen later.
By repeating what was done in Section \ref{sectobs}, 
we get the existence and uniqueness of $\sigma^{1,\varepsilon}$ and $\sigma^{2,\varepsilon}$ by Fredholm alternative. 
Furthermore, $\sigma^{1,\varepsilon}$ and $\sigma^{2,\varepsilon}$ are well defined
and $\sigma^{1,\varepsilon}, \sigma^{2,\varepsilon} \in C^{\infty} (U \setminus \{ x_0 \})$.
In order to derive further properties of $\sigma^{1,\varepsilon}$ and $\sigma^{2,\varepsilon}$, 
we need the following useful formulas.
\begin{Lemma}
For every $\varphi_1, \varphi_2 \in C^2 (\overline{U})$ we have
\begin{equation} \label{gbh} \begin{split}
(2 - i) \varphi_1 (x_0) 
&= - \varepsilon \int_{\partial U} \frac{\partial \sigma^{1,\varepsilon} }{\partial \nu} \varphi_1 \, dS
- \int_U (\gamma^\varepsilon)' \mid_{\theta^\varepsilon_2} \varphi_1 \, \sigma^{2,\varepsilon} \, dx \\
& + \int_U \left[ (1 + (\gamma^\varepsilon)' \mid_{\theta^\varepsilon_1}) \, \varphi_1
+ D_pH_1 \cdot D \varphi_1 - \varepsilon \Delta \varphi_1 \right] \, \sigma^{1,\varepsilon} \, dx,
\end{split} \end{equation}
and
\begin{equation} \label{gbh2} \begin{split}
( i - 1 ) \varphi_2 (x_0) 
&= - \varepsilon \int_{\partial U} \frac{\partial \sigma^{2,\varepsilon} }{\partial \nu} \varphi_2 \, dS 
- \int_U (\gamma^\varepsilon)' \mid_{\theta^\varepsilon_1} \varphi_2 \, \sigma^{1,\varepsilon} \, dx \\
& + \int_U \left[ (1 + (\gamma^\varepsilon)' \mid_{\theta^\varepsilon_2}) \, \varphi_2
+ D_pH_2 \cdot D \varphi_2 - \varepsilon \Delta \varphi_2 \right] \, \sigma^{2,\varepsilon} \, dx,
\end{split} \end{equation}
where $\nu$ is the outer unit normal to $\partial U$.
\end{Lemma}
\begin{proof}
The conclusion follows by simply multiplying by $\varphi_j$ $(j=1,2)$ 
the two equations in \eqref{rfd} and integrating by parts.
\end{proof}
{From the previous lemma, the analogous of Lemma \ref{obs_lem2} follows}.  
\begin{Lemma}[Properties of $\sigma^{1,\varepsilon}, \sigma^{2,\varepsilon}$] \label{kiUJ}
Let $\nu$ be the outer unit normal to $\partial U$. Then
\begin{itemize}
\item[(i)] $\sigma^{j,\varepsilon} \ge 0$ on $\overline{U}$.  
In particular, $\dfrac{\partial \sigma^{j,\varepsilon}}{\partial \nu} \le 0$ on $\partial U \,\, (j=1,2)$.
\item[(ii)] The following equality holds: 
\begin{equation}
\sum_{j=1}^2 \left( \int_U  \sigma^{j,\varepsilon} \,dx 
- \varepsilon \int_{\partial U} \dfrac{\partial \sigma^{j,\varepsilon}}{\partial \nu} \,dS \right) = 1 .
\notag
\end{equation}
In particular,
$$
\sum_{j=1}^2 \int_U  \sigma^{j,\varepsilon} \,dx  \le 1.
$$
\end{itemize}
\end{Lemma}

We are now able to prove a uniform bound on $u^{\varepsilon}_1$ and $u^{\varepsilon}_2$.
{The proof is skipped, since it is analogous to those of the previous sections.}
\begin{Lemma} \label{bd}
There exists a positive constant $C$, independent of $\varepsilon$, such that
\begin{equation*}
\| u^{\varepsilon}_1 \|_{L^{\infty}},  \| u_2^{\varepsilon} \|_{L^{\infty}} \leq C.
\end{equation*}
\end{Lemma}
Next lemma will be used to give a uniform bound for $D u^{\varepsilon}_1$ and $D u^{\varepsilon}_2$.
\begin{Lemma} \label{gamprimeb}
We have 
\begin{equation*}
\int_U ( \gamma^\varepsilon )'  \mid_{\theta^1_{\varepsilon}} \, \sigma^{1,\varepsilon} \, dx
+ \int_U ( \gamma^\varepsilon )'  \mid_{\theta^2_{\varepsilon}} \, \sigma^{2,\varepsilon} \, dx
\leq C,
\end{equation*}
where $C$ is a positive constant independent of $\varepsilon$.
\end{Lemma}

\begin{proof}
First of all, observe that condition (H\ref{kilo}.1) implies that
\begin{equation} \label{conv12}
H_j (x,p) - D_p H_j (x,p) \cdot p \leq H_j (x,0), \qquad \text{ for every }(x,p) \in \overline{U} \times \mathbb{R}^n, \qquad j=1,2.
\end{equation}
In the same way, the convexity of $\gamma$ implies
\begin{equation} \label{convgam}
 \gamma^{\varepsilon} (s) - \left[ (\gamma^{\varepsilon})' (s) \right] s 
= \gamma \left( \frac{s}{\varepsilon} \right) - 
\left[ (\gamma') \left( \frac{s}{\varepsilon} \right) \right] \frac{s}{\varepsilon} \leq \gamma (0) = 0.
\end{equation}
Equation \eqref{tg}$_1$ gives
\begin{equation*} \begin{split}
0&=u^\varepsilon_1 + H_1 (x,Du_1^\varepsilon) 
+\gamma^\varepsilon \mid_{\theta^1_{\varepsilon}} - \varepsilon \Delta u_1^\varepsilon \\
&= u^\varepsilon_1 + D_p H_1 (x,Du^\varepsilon_1) \cdot D u^{\varepsilon}_1
- \varepsilon \Delta u^\varepsilon_1 +H_1 (x,Du^\varepsilon_1) - D_p H_1 (x,Du^\varepsilon_1) \cdot D u^{\varepsilon}_1 \\
& \hspace{.3cm}  +\gamma^\varepsilon \mid_{\theta^1_{\varepsilon}} 
- ( \gamma^\varepsilon )' \mid_{\theta^1_{\varepsilon}} \theta^1_{\varepsilon}
+ ( \gamma^\varepsilon )' \mid_{\theta^1_{\varepsilon}} (u^{\varepsilon}_1 - u_2^{\varepsilon}) - ( \gamma^\varepsilon )'  \mid_{\theta^1_{\varepsilon}} \psi_1.
\end{split} \end{equation*}
Multiplying last relation by $\sigma^{1,\varepsilon}$, integrating and using \eqref{conv12} and \eqref{convgam}
\begin{equation*} \begin{split}
&\int_U ( \gamma^\varepsilon )'  \mid_{\theta^1_{\varepsilon}} \psi_1 \, \sigma^{1,\varepsilon} \, dx
= \int_U \left[ H_1 (x,Du^\varepsilon_1) - D_p H_1 (x,Du^\varepsilon_1) \cdot D u^{\varepsilon}_1 \right] \, \sigma^{1,\varepsilon} \, dx \\
&+ \int_U \left[ \gamma^\varepsilon \mid_{\theta^1_{\varepsilon}} 
- ( \gamma^\varepsilon )' \mid_{\theta^1_{\varepsilon}} \theta^1_{\varepsilon} 
\right] \, \sigma^{1,\varepsilon} \, dx \\
&+ \int_U \left[ \left( 1 + ( \gamma^\varepsilon )' \mid_{\theta^1_{\varepsilon}} \right) u^\varepsilon_1 
+ D_p H_1 (x,Du^\varepsilon_1) \cdot D u^{\varepsilon}_1
- \varepsilon \Delta u^\varepsilon_1 
- ( \gamma^\varepsilon )' \mid_{\theta^1_{\varepsilon}} u_2^{\varepsilon}
\right] \, \sigma^{1,\varepsilon} \, dx \\
&\leq \int_U  H_1 (x,0) \, \sigma^{1,\varepsilon} \, dx \\
&+\int_U \left[ \left( 1 + ( \gamma^\varepsilon )' \mid_{\theta^1_{\varepsilon}} \right) u^\varepsilon_1
 + D_p H_1 (x,Du^\varepsilon_1) \cdot D u^{\varepsilon}_1
- \varepsilon \Delta u^\varepsilon_1 - ( \gamma^\varepsilon )' \mid_{\theta^1_{\varepsilon}} u_2^{\varepsilon}
\right] \, \sigma^{1,\varepsilon} \, dx.
\end{split} \end{equation*}
Analogously, 
\begin{equation*} \begin{split}
&\int_U ( \gamma^\varepsilon )'  \mid_{\theta^2_{\varepsilon}} \psi_2 \, \sigma^{2,\varepsilon} \, dx 
\leq \int_U H_2 (x,0)  \,
\sigma^{2,\varepsilon} \, dx \\
&\hspace{1cm}+ \int_U \left[ \left( 1 + ( \gamma^\varepsilon )' \mid_{\theta^2_{\varepsilon}} \right)
u_2^\varepsilon 
+ D_p H_2 (x,Du_2^\varepsilon) \cdot D u_2^{\varepsilon}
- \varepsilon \Delta u_2^\varepsilon - ( \gamma^\varepsilon )' \mid_{\theta^2_{\varepsilon}} u_2^{\varepsilon}
\right] \, \sigma^{2,\varepsilon} \, dx.
\end{split} \end{equation*}
Summing up the last two relations and using \eqref{gbh} and \eqref{gbh2}
\begin{equation*} \begin{split}
& \int_U ( \gamma^\varepsilon )'  \mid_{\theta^1_{\varepsilon}} \psi_1 \, \sigma^{1,\varepsilon} \, dx
+ \int_U ( \gamma^\varepsilon )'  \mid_{\theta^2_{\varepsilon}} \psi_2 \, \sigma^{2,\varepsilon} \, dx 
\leq (2 - i) u^{\varepsilon}_1 (x_0) + (i-1) u^{\varepsilon}_2 (x_0) \\
&\hspace{.3cm}+ \| H_1 (\cdot ,0)\|_{L^{\infty}} \int_U  \, \sigma^{1,\varepsilon} \, dx
+ \| H_2 (\cdot ,0)\|_{L^{\infty}} \int_U  \, \sigma^{2,\varepsilon} \, dx.
\end{split} \end{equation*}
Thus,
\begin{equation*} \begin{split}
& \int_U ( \gamma^\varepsilon )'  \mid_{\theta^1_{\varepsilon}} \, \sigma^{1,\varepsilon} \, dx
+ \int_U ( \gamma^\varepsilon )'  \mid_{\theta^2_{\varepsilon}} \, \sigma^{2,\varepsilon} \, dx 
\leq \frac{2-i}{\alpha} u_1^{\varepsilon} (x_0) + \frac{i-1}{\alpha} u_2^{\varepsilon} (x_0) \\
&\hspace{.3cm}+ \frac{ \| H_1 (\cdot ,0)\|_{L^{\infty}} }{\alpha} \int_U  \, \sigma^{1,\varepsilon} \, dx
+ \frac{ \| H_2 (\cdot ,0)\|_{L^{\infty}} }{\alpha} \int_U  \, \sigma^{2,\varepsilon} \, dx \leq C ,
\end{split} \end{equation*}
where we used \eqref{strpos}, Lemma \ref{kiUJ} and Lemma \ref{bd}.
\end{proof}
We can finally show the existence of a uniform bound for the gradients 
of $u^{\varepsilon}_1$ and $u^{\varepsilon}_2$.
\begin{Lemma} \label{boundgrad}
There exists a positive constant $C$, independent of $\varepsilon$, such that
\begin{equation*}
\| D u^{\varepsilon}_1 \|_{L^{\infty}}, \, \, \| D u^{\varepsilon}_2 \|_{L^{\infty}} \leq C.
\end{equation*}
\end{Lemma}

\begin{proof}

\textbf{Step I: Bound on $\partial U$.}

As it was done in Section \ref{sectobs}, we are going to construct appropriate barriers.
For $\varepsilon$ small enough, assumption (H\ref{kilo}.4) implies that
\begin{equation*} 
\left\{ \begin{aligned}
\Phi_1 +H_1 (x,D\Phi_1 ) 
+\gamma^\varepsilon(\Phi_1 - \Phi_2 - \psi_1 ) 
< \varepsilon \Delta \Phi_1 \quad\mbox{in}~U, \vspace{.05in}\\
\Phi_2 +H_2 (x,D \Phi_2 ) 
+\gamma^\varepsilon(\Phi_2 - \Phi_1 - \psi_2 ) 
< \varepsilon \Delta \Phi_2 \quad\mbox{in}~U, \vspace{.05in}
\end{aligned} \right.
\end{equation*}
and $\Phi_1 = \Phi_2 =0$ on $\partial U$. 
Therefore, $(\Phi_1,\Phi_2)$ is a sub-solution of \eqref{tg}. 
By the comparison principle, $u^\varepsilon_j \ge \Phi_j$ in~$U$, $j=1,2$.

Let $d(x)$, $\delta$, and $U_{\delta}$ be as in the proof of Proposition \ref{obs1}. 
For $ \mu > 0 $ large enough, the uniform bounds of $\|u^\varepsilon_1\|_{L^\infty}$ and $\|u^\varepsilon_2\|_{L^\infty}$ yield
$v := \mu  d \ge u^\varepsilon_j$ on $\partial U_\delta $, $j=1,2$, so that
\begin{equation*} 
\left\{ \begin{aligned}
v +H_1(x,Dv)+\gamma^\varepsilon( v - v - \psi_1 ) - \varepsilon \Delta v 
= v+H_1(x,Dv) - \varepsilon \Delta v
\ge H_1(x,\mu Dd) - \mu C \quad\mbox{in}~U, \vspace{.05in}\\
v +H_2(x,Dv)+\gamma^\varepsilon( v - v - \psi_2 ) - \varepsilon \Delta v
=  v +H_2(x,Dv) - \varepsilon \Delta v
\ge H_2(x,\mu Dd) - \mu C \quad\mbox{in}~U. \vspace{.05in}\\
\end{aligned} \right.
\end{equation*}
Now, we have $\Phi_j= u^\varepsilon_j= v = 0$ on $\partial U$.
Also, thanks to assumption (H\ref{kilo}.2), for $\mu >0$ large enough
\begin{equation*} 
\left\{ \begin{aligned}
v +H_1(x,Dv)+\gamma^\varepsilon(v- v - \psi_1 ) - \varepsilon \Delta v \ge 0 \quad\mbox{in}~U, \vspace{.05in}\\
v +H_2(x,Dv)+\gamma^\varepsilon(v - v - \psi_2 ) - \varepsilon \Delta v \ge 0 \quad\mbox{in}~U, \vspace{.05in}\\
\end{aligned} \right.
\end{equation*}
that is, the pair $(v,v)$ is a super-solution for the system \eqref{tg}.
Thus, the comparison principle gives us that $\Phi_j \le u^\varepsilon_j \le v_j$ in $ U_\delta$.
Then, from the fact that $\Phi_j= u^\varepsilon_j= v = 0$ on $\partial U$ we get 
$$
\dfrac{\partial v}{\partial \nu}(x) \le \dfrac{\partial u^\varepsilon_j}{\partial \nu}(x) \le \dfrac{\partial \Phi_j}{\partial \nu}(x),\quad \text{for } x\in \partial U.
$$
Hence, we obtain $\| Du^{\varepsilon}_j \|_{L^{\infty}(\partial U)} \leq C$, $j=1,2$.

\textbf{Step II: Bound on $U$.}

Assume now that there exists $\widehat{x} \in U$ such that
\begin{equation*}
\max_{\substack{ j=1,2 \\  x \in \overline{U} }} w^{\varepsilon}_j (x) = w^{\varepsilon}_1 ( \widehat{x} ),
\hspace{1.5cm} \text{ where }  
w^{\varepsilon}_j (x):= \frac{1}{2} | D u^{\varepsilon}_j |^2, \quad j=1,2.
\end{equation*}
By a direct computation one can see that
\begin{equation*} 
2(1+(\gamma^\varepsilon) '  \mid_{\theta^{\varepsilon}_1})w^\varepsilon_1 + D_p H_1 \cdot D w^\varepsilon_1 
+ D_xH_1 \cdot Du^\varepsilon_1 -(\gamma^\varepsilon)' \mid_{\theta^{\varepsilon}_1} Du^\varepsilon_1 \cdot ( D\psi_1 +  D u^{\varepsilon}_2 )
 = \varepsilon \Delta w^\varepsilon_1 - \varepsilon |D^2 u^\varepsilon_1|^2.
\end{equation*}
Multiplying last relation by $\sigma^{1,\varepsilon}$ and integrating over $U$
\begin{equation} \label{obs25}
\begin{split}
&2 \int_U w^\varepsilon_1 \sigma^{1,\varepsilon} \, dx 
+ \int_U D_p H_1 \cdot D w^\varepsilon_1 \sigma^{1,\varepsilon} \, dx 
- \varepsilon \int_U \Delta w^\varepsilon_1\sigma^{1,\varepsilon} \, dx 
+ \int_U \varepsilon |D^2 u^\varepsilon_1|^2 \sigma^{1,\varepsilon} \, dx \\
& + \int_U D_x H_1 \cdot Du^\varepsilon_1 \sigma^{1,\varepsilon} \, dx
+ \frac{1}{2} \int_U (\gamma^\varepsilon)' \mid_{\theta^{\varepsilon}_1} 
\left[  | Du^\varepsilon_1 |^2 + | Du^\varepsilon_1 - Du^\varepsilon_2 |^2 
- | Du^\varepsilon_2 |^2 \right] \, \sigma^{1,\varepsilon} \, dx \\
&- \int_U (\gamma^\varepsilon)' \mid_{\theta^{\varepsilon}_1} Du^\varepsilon_1 \cdot D\psi_1 \sigma^{1,\varepsilon} \, dx
 =  0.
\end{split} \end{equation}
Then, using equation \eqref{gbh}$_1$ with $i=1$ and $x_0 = \widehat{x}$
\begin{equation}  \label{obs26} 
\begin{split}
&\int_U w^\varepsilon_1 \sigma^{1,\varepsilon} \, dx 
+ \int_U \varepsilon |D^2 u^\varepsilon_1|^2 \sigma^{1,\varepsilon} \, dx 
- \int_U (\gamma^\varepsilon)' \mid_{\theta^{\varepsilon}_1} Du^\varepsilon_1 \cdot D\psi_1 \sigma^{1,\varepsilon} \, dx \\
& + \int_U D_x H_1 \cdot Du^\varepsilon_1 \sigma^{1,\varepsilon} \, dx
+ \frac{1}{2} \int_U (\gamma^\varepsilon)' \mid_{\theta^{\varepsilon}_1} 
\left[  | Du^\varepsilon_1 - Du^\varepsilon_2 |^2 
- | Du^\varepsilon_2 |^2 \right] \, \sigma^{1,\varepsilon} \, dx \\
& w^\varepsilon_1 ( \widehat{x}) 
+ \varepsilon \int_{\partial U} \frac{\partial \sigma^{1,\varepsilon} }{\partial \nu} w^\varepsilon_1 \, dS
+ \int_U (\gamma^\varepsilon)' \mid_{\theta^\varepsilon_2} w^\varepsilon_1 \, \sigma^{2,\varepsilon} \, dx =  0,
\end{split} \end{equation}
which implies
\begin{equation*} %\label{obs2} 
\begin{split}
& w^\varepsilon_1 ( \widehat{x})  - \int_U (\gamma^\varepsilon)' \mid_{\theta^{\varepsilon}_1} w^\varepsilon_2 \, \sigma^{1,\varepsilon} \, dx  
+ \int_U (\gamma^\varepsilon)' \mid_{\theta^\varepsilon_2} w^\varepsilon_1 \, \sigma^{2,\varepsilon} \, dx\\
& \leq \int_U (\gamma^\varepsilon)' \mid_{\theta^{\varepsilon}_1} Du^\varepsilon_1 \cdot D\psi_1 \sigma^{1,\varepsilon} \, dx
- \int_U D_x H_1 \cdot Du^\varepsilon_1 \sigma^{1,\varepsilon} \, dx 
- \varepsilon \int_{\partial U} \frac{\partial \sigma^{1,\varepsilon} }{\partial \nu} w^\varepsilon_1 \, dS.
\end{split} \end{equation*}
Let now $\eta > 0$ be a constant to be chosen later.
Using Step I and Lemmas \ref{kiUJ} and \ref{gamprimeb}, thanks to Young's inequality
\begin{equation} \label{fc1} 
\begin{split}
& w^\varepsilon_1 ( \widehat{x} )  - \int_U (\gamma^\varepsilon)' \mid_{\theta^{\varepsilon}_1} w^\varepsilon_2 \, \sigma^{1,\varepsilon} \, dx  
+ \int_U (\gamma^\varepsilon)' \mid_{\theta^\varepsilon_2} w^\varepsilon_1 \, \sigma^{2,\varepsilon} \, dx\\
& \leq \int_U (\gamma^\varepsilon)' \mid_{\theta^{\varepsilon}_1} 
\left[  \eta^2 w^\varepsilon_1 (\widehat{x}) + \frac{\|D\psi_1\|^2_{L^{\infty}}}{2 \, \eta^2} 
\right] \sigma^{1,\varepsilon} \, dx  \\
&\hspace{.4cm} + \int_U \left[  \eta^2 w^\varepsilon_1 (\widehat{x}) + \frac{\|D_x H_1\|^2_{L^{\infty}}}{2 \, \eta^2} 
\right] \sigma^{1,\varepsilon} \, dx +C \\
&\leq \eta^2 (C + 1) w^\varepsilon_1 (\widehat{x}) + C \left( 1 + \frac{1}{\eta^2} \right).
\end{split} \end{equation}
In the same way, considering the analogous of equation \eqref{obs25} for the function $w^{\varepsilon}_2$
(recalling that in \eqref{gbh} we chose $i=1$) we can obtain the following inequality :
\begin{equation} \label{fc2}
- \int_U (\gamma^\varepsilon)' \mid_{\theta^{\varepsilon}_2} w^\varepsilon_1 \, \sigma^{2,\varepsilon} \, dx  
+ \int_U (\gamma^\varepsilon)' \mid_{\theta^\varepsilon_1} w^\varepsilon_2 \, \sigma^{1,\varepsilon} \, dx
\leq \eta^2 (C + 1) w^\varepsilon_1 (\widehat{x}) + C \left( 1 + \frac{1}{\eta^2} \right), 
\end{equation}
where we also used the fact that $\| w^{\varepsilon}_2 \|_{L^{\infty}} \leq w^\varepsilon_1 (\widehat{x})$.
Summing inequalities \eqref{fc1} and \eqref{fc2} and choosing $\eta >0$ small enough the conclusion follows.
\end{proof}
Next lemma gives a control of the Hessians $D^2 u^{\varepsilon}_1$ 
and  $D^2 u^{\varepsilon}_2$ in the support of $\sigma^{1,\varepsilon}$ and $\sigma^{2,\varepsilon}$
respectively.
\begin{Lemma} \label{dggt}
There exists a positive constant $C$, independent of $\varepsilon$, such that
\begin{equation*}
\sup_{j=1,2} \int_U \varepsilon |D^2 u^\varepsilon_j|^2 \sigma^{j,\varepsilon} \, dx \leq C.
\end{equation*}
\end{Lemma}
\begin{proof}
The bound of the Hessian of $D^2 u^{\varepsilon}_1$ comes from identity \eqref{obs26}, together with Lemma \ref{boundgrad}.
The other bound can be obtained in an similar way.
\end{proof}
We can finally prove the analogous of Lemma \ref{obs_lem4}.
\begin{Lemma}
There exists a positive constant, independent of $\varepsilon$, such that
\begin{equation*}
\max_{\substack{ j=1,2 \\  x \in \overline{U} }} \dfrac{\theta^{\varepsilon}_j (x)}{\varepsilon} \leq C, 
\quad \quad  \max_{\substack{ j=1,2 \\  x \in \overline{U} }} \gamma^{\varepsilon} ( \theta^{\varepsilon}_j (x) ) \leq C.
\end{equation*}
\end{Lemma}

\begin{proof}
It will be enough to prove the second inequality, since the first one will follow by the definition of $\gamma^{\varepsilon}$.
If the maximum is attained at the boundary, then 
\begin{equation*}
\max_{\substack{ j=1,2 \\  x \in \overline{U} }} \gamma^{\varepsilon} ( \theta^{\varepsilon}_j (x) ) 
= \max_{\substack{ j=1,2 \\  x \in \partial U }} \gamma^{\varepsilon} ( - \psi_j (x) ) =0.
\end{equation*}
Otherwise, let us assume that there exists $x_1 \in U$ such that 
\begin{equation*}
\max_{j=1,2} \max_{x \in \overline{U}} \gamma^\varepsilon (\theta^\varepsilon_j)= 
\gamma^\varepsilon (\theta^\varepsilon_1) (x_1) > 0, \quad \quad \gamma^{\varepsilon} (\theta^{\varepsilon}_2 (x_1)) = 0.
\end{equation*}
Since $\gamma^\varepsilon$ is increasing and $\gamma^{\varepsilon} (z) > 0$ if and only if $z > 0$, we also have 
$\max_{x \in \overline{U}} (\theta^\varepsilon_1(x)) = \theta^\varepsilon_1 (x_1) > 0$.
Evaluating the two equations in \eqref{tg} at $x_1$ and subtracting the second one from the first one
\begin{align*} 
&\theta^\varepsilon_1 (x_1) + \gamma^\varepsilon (\theta^\varepsilon_1 (x_1) ) 
= \varepsilon \Delta u^{\varepsilon}_1 (x_1) - \varepsilon \Delta u^{\varepsilon}_2 (x_1) 
- H_1 (x_1, D u_1^{\varepsilon} (x_1)) + H_2 (x_1, D u_2^{\varepsilon} (x_1)) - \psi_1 (x_1)  \\
& \le \varepsilon \Delta \psi_1 (x_1)  - H_1 (x_1, D u_1^{\varepsilon} (x_1)) + H_2 (x_1, D u_2^{\varepsilon} (x_1)) - \psi_1 (x_1) \\
&\leq \| \Delta \psi_1 (\cdot) \|_{L^{\infty}}  + \| H_1 (\cdot, D u_1^{\varepsilon} (\cdot)) \|_{L^{\infty}} 
+ \| H_2 (\cdot, D u_2^{\varepsilon} (\cdot)  ) \|_{L^{\infty}}  + \| \psi_1 (\cdot ) \|_{L^{\infty}} \leq C,
\end{align*}
where we the last inequality follows from Lemma \ref{boundgrad}.
\end{proof}
We now set for every $\varepsilon \in (0,1)$
\begin{equation*}
u_{j,\varepsilon}^\varepsilon (x) := \frac{\partial u^{\varepsilon}_j}{\partial \varepsilon} (x),
\quad x \in \overline{U}, j=1,2.
\end{equation*}
The next lemma gives a uniform bound for $u_{1,\varepsilon}^\varepsilon$
and $u_{2,\varepsilon}^\varepsilon$, thus concluding the proof of Theorem~\ref{obs_speedfin}.
\begin{Lemma}
There exists a positive constant $C > 0$ such that
\begin{equation*}
\max_{\substack{ j=1,2 \\  x \in \overline{U} }}  | u_{j,\varepsilon}^\varepsilon (x) | \leq \frac{C}{\varepsilon^{1/2}}.
\end{equation*}
\end{Lemma}
\begin{proof}
If the above maximum is attained at the boundary, then 
\begin{equation*}
\max_{\substack{ j=1,2 \\  x \in \overline{U} }}  | u_{j,\varepsilon}^\varepsilon (x) |=
\max_{\substack{ j=1,2 \\  x \in \partial U }}  | u_{j,\varepsilon}^\varepsilon (x) | = 0, 
\end{equation*}
since $u_{1,\varepsilon}^\varepsilon = u_{2,\varepsilon}^\varepsilon = 0$ on $\partial U$.
Otherwise, assume that there exists $\overline{x} \in U$ such that
\begin{equation*}
\max_{\substack{ j=1,2 \\  x \in \overline{U} }}  | u_{j,\varepsilon}^\varepsilon (x) |
=  | u_{1,\varepsilon}^\varepsilon (\overline{x}) |.
\end{equation*}
Differentiating \eqref{tg} w.r.t. $\varepsilon$ we have 
\begin{equation} \label{tg4}
\left\{ \begin{aligned}
(1+(\gamma^\varepsilon)' \mid_{\theta^{\varepsilon}_1} )u_{1,\varepsilon}^\varepsilon 
+D_p H_1\cdot Du_{1,\varepsilon}^\varepsilon 
- (\gamma^\varepsilon)' \mid_{\theta^{\varepsilon}_1} u_{2,\varepsilon}^\varepsilon
+ \gamma_\varepsilon^\varepsilon \mid_{\theta^{\varepsilon}_1}
= \varepsilon \Delta u_{1,\varepsilon}^\varepsilon 
+ \Delta u^\varepsilon_1 \quad\mbox{in}~U, \vspace{.05in}\\
(1+(\gamma^\varepsilon)' \mid_{\theta^{\varepsilon}_2} )u_{2,\varepsilon}^\varepsilon 
+D_p H_2\cdot Du_{2,\varepsilon}^\varepsilon 
- (\gamma^\varepsilon)' \mid_{\theta^{\varepsilon}_2} u_{1,\varepsilon}^\varepsilon
+ \gamma_\varepsilon^\varepsilon \mid_{\theta^{\varepsilon}_2}
= \varepsilon \Delta u_{2,\varepsilon}^\varepsilon + \Delta u^\varepsilon_2 \quad\mbox{in}~U. \vspace{.05in}
\end{aligned} \right.
\end{equation}
Let $\sigma^{1,\varepsilon}$ and $\sigma^{2,\varepsilon}$ be the solutions 
to system \eqref{rfd} with $i=1$ and $x_0 = \overline{x}$.
Multiplying \eqref{tg4}$_1$ and \eqref{tg4}$_2$ by $\sigma^{1,\varepsilon}$
and $\sigma^{2,\varepsilon}$ respectively, integrating by parts 
and adding up the two relations obtained we have
\begin{equation*}
u_{1,\varepsilon}^\varepsilon (\overline{x})
= \sum_{j=1}^2
\left( \int_U \Delta u^{\varepsilon}_j \, \sigma^{j,\varepsilon} \, dx
-  \int_U \gamma^{\varepsilon}_{\varepsilon} \mid_{\theta^{\varepsilon}_j} \sigma^{j,\varepsilon} \, dx \right).
\end{equation*}
Thus, 
\begin{equation*}
|u_{1,\varepsilon}^\varepsilon (\overline{x})|
\leq \sum_{j=1}^2 
\left( \int_U |\Delta u^{\varepsilon}_j | \, \sigma^{j,\varepsilon} \, dx
+ \int_U | \gamma^{\varepsilon}_{\varepsilon} \mid_{\theta^{\varepsilon}_j} | \, \sigma^{j,\varepsilon} \, dx \right).
\end{equation*}
At this point, the proof can be easily concluded by repeating what was done in Section \ref{sectobs}
showing relations \eqref{obs6}--\eqref{obs8}.
\end{proof}

\end{section}

\bigskip

\bibliographystyle{alpha}
\bibliography{mainbibliography}

\end{document}